\documentclass[12pt]{article}
\usepackage[francais,english]{babel}
\usepackage[T1]{fontenc}
\usepackage{kpfonts}
\usepackage{color}
\usepackage{amssymb}
\usepackage{geometry}
\RequirePackage{filecontents}
\usepackage{amsmath,mathtools}
\usepackage{graphicx}
\usepackage{sectsty}
\usepackage{lipsum}
\usepackage{mathrsfs}
\usepackage{amsthm}
\allsectionsfont{\centering}
\newtheorem{theorem}{Theorem}[section]

\newtheorem{proposition}{Proposition}[section]
\newtheorem{remark}{Remark}[section]
\newtheorem{lemma}{Lemma}[section]

\numberwithin{equation}{section}
\begin{document}
\large
\begin{center}
   \textbf{\large{Bargmann transfoms associated with reproducing kernel Hilbert space\\ and application to Dirichlet spaces}}
\end{center}
\begin{center}
   \textbf{Nour eddine Askour $^{1,2,a}$ and  Mohamed Bouaouid $^{1,b}$}
\end{center}
\begin{center}
$^{1}$ Department of Mathematics, Sultan Moulay Slimane University, Faculty
		of Sciences and Technics, Beni Mellal, BP 523, 23000, Morocco.
\end{center}
\begin{center}
$\&$
\end{center}
\begin{center}
$^{2}$ Department of Mathematics, Mohammed V University, Faculty of Sciences, Rabat, P.O. Box 1014, Morocco.
\end{center}

\begin{center}
$^{a}$ n.askour@usms.ma \hspace{0.25cm} and \hspace{0.25cm}$^{b}$ bouaouidfst@gmail.com
\end{center}
\section*{Abstract}
The aim of the present paper is three folds. For a reproducing kernel Hilbert space $\mathcal{A}$ (R.K.H.S) and a $\sigma-$finite measure space $(M_{1},d\mu_{1})$ for which the corresponding $L^{2}-$space is a separable Hilbert space, we first build an isometry of Bargmann type as an integral transform from $L^{2}(M_{1},d\mu_{1})$ into $\mathcal{A}$. Secondly, in the case where there exists a $\sigma-$finite measure space $(M_{2},d\mu_{2})$ such that the Hilbert space $L^{2}(M_{2},d\mu_{2})$ is separable and $\mathcal{A}\subset L^{2}(M_{2},d\mu_{2})$ the inverse isometry is also given in an explicit form as an integral transform. As consequence, we recover some classical isometries of Bargmann type. Thirdly, for the classical Dirichlet space as R.K.H.S, we elaborate a new isometry of Bargmann type. Furthermore, for this Dirichlet space, we give a new characterization, as harmonic space of a single second order elliptic partial differential operator for which, we present some spectral properties. Finally, we extend the same results to a class of generalized Bergman-Dirichlet space.
\begin{description}
  \item[Keywords:] 	Hilbert spaces with reproducing kernels; Dirichlet spaces; Orthogonal polynomials; Integral transforms; Partial differential operators; Spectral theory.
  \item[2010 Mathematics Subject Classification:] 46E22; 31C25; 33C45; 65R10; 47F05; 58C40.
\end{description}
\newpage
\tableofcontents
\newpage
\section{Introduction}
The classical Bargmann transform is an integral transform between the Schr\"{o}dinger space of square integrable functions and the Fock space of holomorphic functions. It was introduced by Bargmann \cite{Bar} in 1961 giving an isomorphism between two formulations of quantum mechanics
 \begin{align}\label{T1.1}
 \nonumber B\hspace{0.2cm}:L^{2}(\mathbb{R},& \hspace{0.2cm}dx)\longrightarrow \mathcal{A}^{2}(\mathbb{C})\subset L^{2}(\mathbb{C},
 \hspace{0.2cm}e^{-\mid z\mid^{2}}d\lambda(z))\\
 &\varphi\longmapsto B[\varphi](z):=\pi^{-\frac{3}{4}}\int_{-\infty}^{+\infty}\exp(-\frac{x^{2}}{2}+\sqrt{2}xz-\frac{z^{2}}{2})\varphi(x)dx,
 \end{align}
 where  $\mathcal{A}^{2}(\mathbb{C})$ is the Bargmann-Fock space  of the holomorphic functions integrable with respect to the Gaussian measure $e^{-\mid z\mid^{2}}d\lambda(z)$. Here, $d\lambda(z)$ is the ordinary area measure. The inner product in $\mathcal{A}^{2}(\mathbb{C})$ is inherited from $L^{2}(\mathbb{C},\hspace{0.2cm}e^{-\mid z\mid^{2}}d\lambda(z))$.\\
From then on, the Bargmann transform becames a powerful tool in mathematical analysis and mathematical physics. For this transform, several generalizations, using different approaches, have been given. Based on the fact that the Bargmann transform arises as the isometry part in the polar decomposition of the restriction operator, from holomorphic functions on complex space to functions on real space, Brian Hall has generalized the Bargmann transform to the compact Lie group \cite{Hal4,Hal5}. In the same perspective, we can also quote the following reference \cite{Ola}. The compact case was also generalized to the case of a special class of compact symmetric spaces by \cite{Hal2}. In the same way, Matthew Stenzel has given a straight forward generalization of the Bargmann transform to the special class of complex Riemann manifolds including compact symmetric spaces \cite{Ste}. In the case of complex type, the Bargmann transform associated with non compact symmetric spaces was also discussed by Brian Hall \cite{Hal1}. Here, in this work, based on the idea that the Bargmann space is a reproducing kernel Hilbert space (R.K.H.S) \cite{} and the integral kernel of the Bargmann transform appear as a generating function of the Hermite polynomials which is an orthogonal system of the Hilbert space $L^{2}(\mathbb{R},e^{-x^{2}}dx)$, we purpose a simple generalization of the Bargmann transform in the frame of R.K.H.S theory. Precisely, we will be concerned with a separable R.K.H.S $\mathcal{A}$ which will plays the same role as the Bargmann space $\mathcal{A}^{2}(\mathbb{C})$ and a separable Hilbert space $H=L^{2}(M_{1},d\mu_{1})$ which will plays the same role of the space $L^{2}(\mathbb{R},e^{-x^{2}}dx)$. More precisely, we have the following result.
\begin{theorem}\label{E11} Let $M_{2}$ be an arbitrary (non-empty) set and let $\mathcal{F}(M_{2})$ the space of all complex valued functions on $M_{2}$. Let $(M_{1},d\mu_{1})$ be a measure space for which the corresponding $L^{2}(M_{1},d\mu_{1})$ space is a separable Hilbert space and $\mathcal{A}\subset \mathcal{F}(M_{2})$ being a reproducing kernel Hilbert space (R.K.H.S). If $\mathcal{B}_{1}=(\varphi_{j})_{j\in\mathbb{Z}^{+}}$ and $\mathcal{B}_{2}=(\psi_{j})_{j\in\mathbb{Z}^{+}}$ be two orthonormal basis of $L^{2}(M_{1},d\mu_{1})$ and $\mathcal{A}$, respectively, then the following statements holds.

\begin{description}
  \item[i)] For each fixed $z\in M_{2}$, the series
  \begin{align}\label{E1.1}
  &K(z,w)=\displaystyle{\sum_{j\in\mathbb{Z}^{+}}}\overline{\varphi_{j}(w)}\psi_{j}(z),\hspace{0.25cm}w\in M_{1},
  \end{align}
  converges a.e$-d\mu_{1}$.
\item[ii)] The following integral transform
\begin{align}\label{E1.2}
		\nonumber&L^{2}(M_{1},d\mu_{1})\longrightarrow \mathcal{A}\\
		&\varphi\longmapsto B[\varphi](z):=\int_{M_{1}}K(z,w)\varphi(w)d\mu_{1}(w),\hspace{0.25cm}z\in M_{2},
\end{align}
defines a unitary isomorphism from $L^{2}(M_{1},d\mu_{1})$ into the R.K.H.S $\mathcal{A}$.
\end{description}
Moreover, if the measure $d\mu_{1}$ is a $\sigma-$finite and the set $M_{2}$ is also equipped with a  $\sigma-$finite measure $d\mu_{2}$ such that the R.K.H.S $\mathcal{A}$ is a subspace of $L^{2}(M_{2},d\mu_{2})$. Then, the inverse isomorphism $B^{-1}$ is also given by the following integral transform
 \begin{align}\label{E1.3}
		\nonumber&\mathcal{A} \longrightarrow L^{2}(M_{1},d\mu_{1})\\
	&\psi\longmapsto B^{-1}[\psi](w):=\int_{M_{2}}\overline{K(z,w)}\psi(z)d\mu_{2}(z),
\hspace{0.25cm}w\in M_{1}.
\end{align}
\end{theorem}
There are many ways of computing the integral kernel appearing in the Bargmann transform. In particular, this involves the Blattner-Kostant-Sternberg theory \cite{Ola}. An other way is to use theory of Hermite functions and their generating function \cite{Bar,Zhu1}. Based on this last idea, various authors have introduced new isometries considered as transformations of Bargmann type. For more details, we cite \cite{Elk,Elw,Zou3,Zou1,Zou4,Zou2}.\\
A closed inspection shows that the last cited authors have used the same technical method to deal with case by case theirs isometries of Bargmann type. However, these cases can be brought in an unified framework recovering the classical analysis, $q-$analysis as well as the quaternionic analysis.  This is the purpose of the first part of the present work. Concretely, we first elaborate a general result giving an isometry of Bargmann type and its inverse, using simple geometric intrinsic properties of Hilbert space. This is the subject of the above theorem $(\ref{T1.1})$.\\
 In one hand, these theorems allowed us to recover the results of the last authors being cited above. Effectively, to not overloud the text, we will restrict our treatment to the following three examples: the classical Bargmann transform \cite{Bar}, the second Bargmann transform \cite{Bar} and the generalized Bargmann transform \cite{Elw}.\\
On the other hand, theorem  $(\ref{T2.2})$, with the help of some technical summability methods, allowed us to define and construct Bargmann transforms associated with the Dirichlet space and the generalized Bergmann-Dirichlet space. Parallel to the construction of these Bargmann transforms corresponding to the Dirichlet space and to generalized  Dirichlet space, we characterize these spaces as harmonic spaces of some elliptic second order partial differential operators.\\
 Before going ahead, we give a concise picture of  theses characterizations.\\
 We recall that the Dirichlet space (\cite{Elf}) $\mathcal{D}$ on the unit disk $\mathbb{D}$ is the vector space of holomorphic functions on $\mathbb{D}$ for which the Dirichlet integral
 \begin{align}\label{E1.5}
\mathbf{D}(f)=\frac{1}{\pi}\int_{\mathbb{D}}\mid f^{'}(z)\mid^{2}d\lambda(z),
\end{align}
is finite, where $d\lambda(z)$ is the Lebesgue area measure on $\mathbb{D}$. For this Dirichlet space, we prove the following characterization
\begin{align}\label{E1.6}
 \mathcal{D}=\{F\in D(\tilde{\Delta}),\hspace{0.25cm}\tilde{\Delta}F=0\},
 \end{align}
where $\tilde{\Delta}$ is the partial differential operator defined by
\begin{align}\label{E1.7}
 \tilde{\Delta}=-4(1-\mid z\mid^{2})[(1-\mid z\mid^{2})\frac{\partial^{2}}{\partial z\partial\overline{z}}-2\overline{z}\frac{\partial}{\partial\overline{z}}],
 \end{align}
 acting on the Hilbert space $L^{2}(\mathbb{D},d\lambda(z))$, with the dense domain
 \begin{align}\label{E1.8}
 D(\tilde{\Delta}):=\{F\in L^{2}(\mathbb{D},d\lambda(z)),\hspace{0.25cm}\tilde{\Delta}F\in L^{2}(\mathbb{D},d\lambda(z))\hspace{0.25cm}and\hspace{0.25cm}\frac{\partial F}{\partial z}\in L^{2}(\mathbb{D},d\lambda(z))\}.
 \end{align}
 Also, for the Dirichlet space viewed as Hilbert space endowed with the following scalar product
 \begin{align}\label{E1.9}
 <f,g>_{\mathcal{D}}=f(0)\overline{g(0)}+\int_{\mathbb{D}}f^{'}(z)\overline{g^{'}(z)}d\lambda(z),
 \end{align}
we establishe the following Bargmann transform
\begin{align}\label{E1.10}
 \nonumber L^{2}(\mathbb{R}_{+},& dx)\longrightarrow \mathcal{D}\\
 &f\longmapsto B[f](z):=\int_{0}^{+\infty}K(z,x)f(x)dx,
 \end{align}
  where $K(z,x)$ is given by
  \begin{align}\label{E1.11}
 K(z,x)=\frac{e^{\frac{x}{2}}}{\sqrt{\pi}}[1+\frac{z}{\Gamma(\frac{3}{2})}\int_{0}^{+\infty}\sqrt{t}e^{-t}
 (1-ze^{-t})^{-2}\exp(\frac{-xze^{-t}}{1-ze^{-t}})L_{1}(\frac{x}{1-ze^{-t}})dt].
   \end{align}
In the same way, we deal with the generalized Bergmann-Dirichlet space (\cite{Elh}) $\mathcal{D}_{m}^{\alpha}$, $\alpha>-1$,
 $m\in\mathbb{Z}_{+}$ defined as the space of analytic function $f(z)$ on the unit disk $\mathbb{D}$
\begin{align}\label{E1.12}
\displaystyle{f(z)=\sum_{j=0}^{+\infty}}a_{j}z^{j},\hspace{0.25cm} z\in \mathbb{D},
\end{align}
for which the generalized Dirichlet integral
\begin{align}\label{E1.13}
\mathcal{D}_{m}^{\alpha}(f)=\frac{1}{\pi}\int_{\mathbb{D}}\mid f^{(m)}(z)\mid^{2} d\mu_{\alpha}(z),
\end{align}
is finite, where $d\mu_{\alpha}(z)=(1-\mid z\mid^{2})^{\alpha}d\lambda(z).$\\
Any function $f\in \mathcal{D}_{m}^{\alpha}$ splits as
\begin{align}\label{E1.14}
f(z)=f_{1,m}(z)+f_{2,m}(z),
\end{align}
where
$\displaystyle{f_{1,m}(z)=\sum_{j=0}^{m-1}a_{j}z^{j}}$ and
 $\displaystyle{f_{2,m}(z)=\sum_{j=m}^{+\infty}a_{j}z^{j}},$
with the convention that $f_{1,m}(z)=0$  when $m=0.$\\
The generalized Bergmann-Dirichlet space equipped with the following scalar product
\begin{align}\label{E1.15}
<f,g>_{\alpha,m}=<f_{1,m},g_{1,m}>_{\alpha}+<f_{2,m},g_{2,m}>_{\alpha},
\end{align}
where
\begin{align}\label{E1.16}
<f,g>_{\alpha}=\int_{\mathbb{D}}f(z)\overline{g(z)}d\mu_{\alpha}(z),
\end{align}
is R.K.H.S \cite{Elh}.\\
For this functional space $\mathcal{D}_{m}^{\alpha}$, we prove the following characterization

\begin{align}\label{E1.17}
 \mathcal{D}_{m}^{\alpha}=\{F\in D(\tilde{\Delta}_{\alpha}),\hspace{0.25cm}\tilde{\Delta}_{\alpha}F=0\},
 \end{align}
where  $\tilde{\Delta}_{\alpha}$ is the elliptic partial differential operator defined by
\begin{align}\label{E1.18}
 \tilde{\Delta}_{\alpha}:=-4(1-\mid z\mid^{2})[(1-\mid z\mid^{2})\frac{\partial^{2}}
 {\partial z\partial\overline{z}}-(\alpha+2)\overline{z}\frac{\partial}{\partial\overline{z}}],
 \end{align}
acting on the Hilbert space $L^{2,\alpha}(\mathbb{D})=L^{2}(\mathbb{D},(1-\mid z\mid^{2})^{\alpha}d\lambda(z))$,
 with the dense domain
 \begin{align}\label{E1.19}
 D(\tilde{\Delta}_{\alpha}):=\{F\in L^{2,\alpha}(\mathbb{D}),\hspace{0.25cm}\tilde{\Delta}_{\alpha}F\in L^{2,\alpha}(\mathbb{D})\hspace{0.25cm}and\hspace{0.25cm}\frac{\partial^{m} F}
 {\partial z^{m}}\in L^{2,\alpha}(\mathbb{D})\}.
 \end{align}
For this generalized Bergmann-Dirichlet space, we establish the following Bargmann transform
\begin{align}\label{E1.20}
 \nonumber L^{2}(\mathbb{R}_{+},& dx)\longrightarrow \mathcal{D}_{m}^{\alpha}\\
 &f\longmapsto B[f](z):=\int_{0}^{+\infty}K(z,x)f(x)dx,
 \end{align}
  where
 \begin{align}\label{E1.21}
  \nonumber K(z,x)&=\frac{x^{\frac{-\alpha}{2}}e^{\frac{x}{2}}}{\sqrt{\pi \Gamma(1+\alpha)}}\displaystyle{\sum_{0\leq j<m}}z^{j}L_{j}^{(\alpha)}(x)\\
   &\hspace{-1.8cm}+\frac{m!z^{m}(\Gamma(\frac{3}{2}))^{-m}\Gamma(\frac{1}{2}))^{1-m}x^{\frac{-\alpha}{2}}
   e^{\frac{x}{2}}}{\sqrt{\pi \Gamma(1+\alpha)}}
   \int_{0}^{+\infty}\omega_{(\alpha,m)}(t)(1-ze^{-t})^{-\alpha-m-1}
 \exp(\frac{-xze^{-t}}{1-ze^{-t}})L_{m}^{(\alpha)}(\frac{x}{1-ze^{-t}})dt.
   \end{align}
and $\varpi_{(\alpha,m)}$ is the function defined as follows
\begin{align}\label{E1.22}
\varpi_{(\alpha,m)}(t)=\sqrt{t}e^{-t}*[(\sqrt{t}e^{-2t})*(\frac{e^{-(\alpha+2)t}}{\sqrt{t}})]*...
*[(\sqrt{t}e^{-mt})*(\frac{e^{-(\alpha+m)t}}{\sqrt{t}})], \hspace{0.25cm}m\geq2.
\end{align}
The notation $f*g$ means the following convolution product
\begin{align}\label{E1.23}
f*g(x)=\int_{0}^{x}f(x-y)g(y)dy.
\end{align}
This paper is summarized as follows. In section 2, we build in a general way a Bargmann transform associated with R.K.H.S and we apply this to recover the classical Bargmann transform on $\mathbb{C}$ (\cite{Bar}), the second Bargmann transform on the unit disk \cite{Bar} and the generalized second Bargmann transforms \cite{Elw}.
  In the section 3, we prove the proposition $(\ref{P3.1})$ related to the characterization of the Dirichlet space as harmonic
  space of a single elliptic second order partial differential operator. Also, in this
  section we prove the proposition $(\ref{P3.3})$ in which we give the Bargmann transform associated with the Dirichlet space.
  The section 4 will be devoted to the generalized Bergman-Dirichlet space. In this section we follow the same lines as in the section 3.
\section{Bargmann transform associated with R.K.H.S}
\subsection{New Bargmann transform associated with R.K.H.S}
In this subsection, based on the reproducing kernel theory \cite{Aro}, we construct a Bargmann transform associated with an abstract reproducing kernel Hilbert space. We have the following result.
\begin{theorem}\label{T2.1} Let $L^{2}(M_{1},d\mu_{1})$ and $L^{2}(M_{2},d\mu_{2})$ be a two separable Hilbert spaces and let $\mathcal{A}\subseteq L^{2}(M_{2},d\mu_{2})$ be a reproducing kernel Hilbert subspace, where $d\mu_{1}$ and $d\mu_{2}$ are a $\sigma-$finite measures. Let $\mathcal{B}_{1}=(\varphi_{j})_{j\in\mathbb{Z}^{+}}$ and $\mathcal{B}_{2}=(\psi_{j})_{j\in\mathbb{Z}^{+}}$ be two orthonormal basis of $L^{2}(M_{1},d\mu_{1})$ and $\mathcal{A}$, respectively,  then we have the following statements
\begin{description}
  \item[i)] For each fixed $z\in M_{2}$, the series
  \begin{align}\label{E2.1}
  K(z,w)=\displaystyle{\sum_{j\in\mathbb{Z}^{+}}}\overline{\varphi_{j}(w)}\psi_{j}(z), \hspace{0.25cm}w\in M_{1},
  \end{align}
  converges a.e$-d\mu_{1}$ and $K(z,.)\in L^{2}(M_{1},d\mu_{1}).$
  \item[ii)] The following integral transform
  \begin{align}\label{E2.2}
		\nonumber&L^{2}(M_{1},d\mu_{1})\longrightarrow \mathcal{A}\\
		&\varphi\longmapsto B[\varphi](z):=\int_{M_{1}}K(z,w)\varphi(w)d\mu_{1}(w),\hspace{0.25cm}z\in M_{2},
		\end{align}
defines a unitary isomorphism from $L^{2}(M_{1},d\mu_{1})$ into the R.K.H.S $\mathcal{A}.$
 \item[iii)] The inverse isomorphism $B^{-1}$ is given by the integral transform
		\begin{align}\label{E2.3}
		\nonumber&\mathcal{A} \longrightarrow L^{2}(M_{1},d\mu_{1})\\
		&\psi\longmapsto B^{-1}[\psi](w):=\int_{M_{2}}\overline{K(z,w)}\psi(z)d\mu_{2}(z),
\hspace{0.25cm}w\in M_{1}.
\end{align}
\item[iv)] We have the following pairwise  formula
   \begin{align}\label{E2.4}
   \hspace{-2cm}\textbf{(a)} \hspace{3cm} \psi_{j}(z)=\int_{M_{1}}K(z,w)\varphi_{j}(w)d\mu_{1}(w),\hspace{0.25cm}z\in M_{2}.
   \end{align}
  \begin{align}\label{E2.5}
  \hspace{-2cm}\textbf{(b)} \hspace{3cm} \varphi_{j}(w)=\int_{M_{2}}\overline{K(z,w)}\psi_{j}(z)d\mu_{2}(z),\hspace{0.25cm}w\in M_{1}.
\end{align}
\end{description}
\end{theorem}
\begin{proof} First, we set $\mathcal{H}_{1}=L^{2}(M_{1},d\mu_{1})$ and $\mathcal{H}_{2}=L^{2}(M_{2},d\mu_{2}).$ Let $\tilde{B}$ be the following linear operator defined by
	\begin{align}\label{E2.6}
			\nonumber&\tilde{B}:\mathcal{H}_{1}\longrightarrow \mathcal{A}\\
		&\varphi=\displaystyle{\sum_{j\in\mathbb{Z}^{+}}}<\varphi,\varphi_{j}>_{\mathcal{H}_{1}}\varphi_{j}\longmapsto \tilde{B}[\varphi]=\displaystyle{\sum_{j\in\mathbb{Z}^{+}}}<\varphi,\varphi_{j}>_{\mathcal{H}_{1}}\psi_{j}.
			\end{align}
By theory of Fourier series in Hilbert space, $\tilde{B}$ is well defined as a  unitary isomorphism from $\mathcal{H}_{1}$ into $\mathcal{A}$. Now, let us prove that  the operator $\tilde{B}$ coincides with the integral operator $B$ defined in $(\ref{E2.2})$.\\
To do so, we first consider the following series
\begin{align}\label{E2.7}
 \tilde{B}[\varphi](z)=\displaystyle{\sum_{j\in\mathbb{Z}^{+}}}<\varphi,\varphi_{j}>_{\mathcal{H}_{1}}\psi_{j}(z).
\end{align}
By the Cauchy-Chwartz inequality, we have
\begin{align}\label{E2.8}
\mid \tilde{B}[\varphi](z)\mid\leq
(\displaystyle{\sum_{j\in\mathbb{Z}^{+}}}\mid<\varphi,\varphi_{j}>_{\mathcal{H}_{1}}\mid^{2})^{\frac{1}{2}}
(\displaystyle{\sum_{j\in\mathbb{Z}^{+}}}\mid\psi_{j}(z)\mid^{2})^{\frac{1}{2}}.
\end{align}
For later use, it is recalled that
\begin{align}\label{E2.9}
\parallel\varphi\parallel_{\mathcal{H}_{1}}=(\displaystyle{\sum_{j\in\mathbb{Z}^{+}}}\mid<\varphi,\varphi_{j}>_{\mathcal{H}_{1}}
			\mid^{2})^{\frac{1}{2}}.
\end{align}
For proving that the series  $\tilde{B}[\varphi](z)$, defined in $(\ref{E2.7})$, is absolutely convergent, it suffice to show that the involved series
\begin{align}\label{E2.10}
\rho(z)=(\displaystyle{\sum_{j\in\mathbb{Z}^{+}}}\mid\psi_{j}(z)\mid^{2})^{\frac{1}{2}},
\end{align}
appearing in the equation $(\ref{E2.8})$, is  convergent.
 Let $K_{\mathcal{A}}(z,w)$ be the reproducing kernel of the Hilbert space $\mathcal{A}$. According to Papadakis theorem \cite[p.12]{Pau}, we have  the following formula
 \begin{align}\label{E2.11}
 K_{\mathcal{A}}(z,w)=\displaystyle{\sum_{j\in\mathbb{Z}^{+}}}f_{j}(z)\overline{f_{j}(w)},\hspace{0.25cm}z,w\in M_{1},
 \end{align}
 for any orthonormal basis $\{f_{j}\}_{j\in\mathbb{Z}^{+}}$ of the Hilbert space $\mathcal{A}$. Moreover, the series in the left hand side of $(\ref{E2.11})$ is pointwise convergent. It follows that the series  $\rho(z)$ defined in  $(\ref{E2.10})$ coincides  with the number $\sqrt{K_{\mathcal{A}}(z,z)}$. Thus, the convergence of the series $\rho(z)$ is obtained.\\
 Returning back to the inequality  $(\ref{E2.8})$ and using $(\ref{E2.9})$ with $(\ref{E2.10})$, we find the following estimate
 \begin{align}\label{E2.12}
			\mid\tilde{B}[\varphi](z)\mid \leq \rho(z)\parallel\varphi\parallel_{\mathcal{H}_{1}}.
\end{align}
 The above inequality traduces the continuity of the following linear functional
 \begin{align}\label{E2.13}
			\nonumber&H_{1}\longrightarrow \mathbb{C}\\
			&\varphi\longmapsto \tilde{B}[\varphi](z),
			\end{align}
for each fixed $z\in M_{2}$.\\
Then, by Riesz-theorem, there exists $h_{z}\in \mathcal{H}_{1}$ such that
\begin{align}\label{E2.14}
\tilde{B}[\varphi](z)=<\varphi,h_{z}>_{\mathcal{H}_{1}}.
\end{align}
From the definition of $\tilde{B}$ given in $(\ref{E2.6})$, we obtain the following equality
\begin{align}\label{E2.15}
 \displaystyle{\sum_{j\in\mathbb{Z}^{+}}}<\varphi,\varphi_{j}>_{\mathcal{H}_{1}}\psi_{j}(z)=
 <\varphi,h_{z}>_{\mathcal{H}_{1}}.
 \end{align}
 Recalling that $\{\varphi_{j}\}_{j\in\mathbb{Z}^{+}}$ is an orthonormal basis of $\mathcal{H}_{1}$ and using the fact that  the series
 \begin{align}\label{E2.16}
 \displaystyle{\sum_{j\in\mathbb{Z}^{+}}}\mid\psi_{j}(z)\mid^{2}=\rho^{2}(z),
 \end{align}
 is convergent, we obtain that the series
 $\displaystyle{\sum_{j\in\mathbb{Z}^{+}}}\overline{\psi_{j}(z)}\varphi_{j}$
converges in $\mathcal{H}_{1}.$ This last result combined with the continuity of hermitian scalar product imply that the equation $(\ref{E2.15})$ can be rewritten as
\begin{align}\label{E2.17} <\varphi,\displaystyle{\sum_{j\in\mathbb{Z}^{+}}}\overline{\psi_{j}(z)}\varphi_{j}>_{\mathcal{H}_{1}}=<\varphi,h_{z}>_{\mathcal{H}_{1}},\hspace{0.25cm}
			for\hspace{0.25cm}all\hspace{0.25cm}
			\varphi\in \mathcal{H}_{1}.
\end{align}
This last  equation leads to the following equation
\begin{align}\label{E2.18}
K(z,w):=\displaystyle{\sum_{j\in\mathbb{Z}^{+}}}\overline{\varphi_{j}}(w)\psi_{j}(z)=\overline{h_{z}}(w),
\hspace{0.25cm}\texttt{a.e}-d\mu_{1}\hspace{0.25cm}in \hspace{0.25cm}M_{1},
\end{align}
for each fixed $z\in M_{2}$. Thus, the first point of the theorem $(\ref{T2.1})$ is proved.\\
Thanks to the equation $(\ref{E2.18})$, the right hand side of the equation $(\ref{E2.14})$ can be written as
\begin{align}\label{E2.19}
<\varphi,h_{z}>_{\mathcal{H}_{1}}=\int_{M_{1}}K(z,w)\varphi(w)d\mu_{1}(w),\hspace{0.25cm}\varphi\in\mathcal{H}_{1} ,\hspace{0.25cm}z\in M_{2}.
\end{align}
Hence, the equation $(\ref{E2.14})$ becomes
\begin{align}\label{E2.20}
\tilde{B}[\varphi](z)=\int_{M_{1}}K(z,w)\varphi(w)d\mu_{1}(w)=B[\varphi](z),\hspace{0.25cm}\varphi\in\mathcal{H}_{1},\hspace{0.25cm}z\in M_{2}.
\end{align}
Thus, this last equality gives the second point of theorem $(\ref{T2.1})$.\\
For ending the proof, it remains to prove that the inverse isomorphism $B^{-1}$ is given by the following integral transform
\begin{align}\label{E2.21}
B^{-1}[\psi](z):=\int_{M_{2}}\overline{K(w,z)}\psi(w)d\mu_{2}(w),\hspace{0.25cm}\psi\in\mathcal{H}_{2},\hspace{0.25cm}z\in M_{1}.
\end{align}
By using the fact that $B$ is a unitary isomorphism, we have
\begin{align}\label{E2.22}
			<B[\varphi_{1}],B[\varphi_{2}]>_{\mathcal{H}_{2}}
=<\varphi_{1},\varphi_{2}>_{\mathcal{H}_{1}},\hspace{0.25cm}\varphi_{1},\hspace{0.2cm}\varphi_{2}\in\mathcal{H}_{1}.
\end{align}
Let $B^{*}:\mathcal{A} \longrightarrow  H_{1}$ be the adjoint operator of $B$ defined by the following equation
			\begin{align}\label{E2.23}
			<B[\varphi],\psi>_{\mathcal{H}_{2}}=<\varphi,B^{*}[\psi]>_{\mathcal{H}_{1}}.
			\end{align}
Using the equation $(\ref{E2.22})$ with $(\ref{E2.23})$, we obtain
\begin{align}\label{E2.24}
<\varphi_{1},B^{*}B(\varphi_{2})>_{\mathcal{H}_{1}}=<\varphi_{1},\varphi_{2}>_{\mathcal{H}_{1}}\hspace{0.25cm}for\hspace{0.2cm} all \hspace{0.25cm} \varphi_{1}\in \mathcal{H}_{1}.
\end{align}
It follows that $B^{*}B=I_{1}$, where $I_{1}$ is the identity of $\mathcal{H}_{1}$, and hence, $B^{-1}=B^{*}.$ By applying the definition of $B^{*}$  and using the Fubini theorem \cite[p.187,p.48]{Ber,Val}, we get
\begin{align}\label{E2.25}
			\nonumber<\varphi,B^{*}[\psi]>_{\mathcal{H}_{1}}&=<B[\varphi],\psi>_{\mathcal{H}_{2}}\\
			\nonumber&=\int_{M_{2}}\int_{M_{1}}K(z,w)\varphi(w)\overline{\psi(z)}d\mu_{1}(w)d\mu_{2}(z)\\
			\nonumber&=\int_{M_{1}}\int_{M_{2}}K(z,w)\varphi(w)\overline{\psi(z)}d\mu_{2}(z)d\mu_{1}(w)\\
			&=\int_{M_{1}}\varphi(w)\overline{\int_{M_{2}}\overline{K(z,w)}\psi(z)d\mu_{2}(z)}d\mu_{1}(w).
\end{align}
Finally, we find
\begin{align}\label{E2.26}
			B^{-1}[\psi](w)=B^{*}[\psi](w)=\int_{M_{2}}\overline{K(z,w)}\psi(z)d\mu_{2}(z),\hspace{0.25cm} w\in M_{1}.
\end{align}
Taking the operator $\tilde{B}$ defined in $(\ref{E2.6})$ and the formula $(\ref{E2.20})$, we obtain easily the first point $(a)$ in $(iv)$. Also, by using the fact that the operator $\tilde{B}$ defined in   $(\ref{E2.46})$ is a one to one isometry, we can write the operator $\tilde{B}^{-1}$ as follows
\begin{align}\label{E2.27}
		\tilde{B}^{-1}: \hspace{0.2cm}\nonumber&\mathcal{A} \longrightarrow \mathcal{H}_{1}\\
		&\psi\longmapsto \tilde{B}^{-1}[\psi]=\sum_{j\in \mathbb{Z}_{+}}<\psi,\psi_{j}>_{\mathcal{H}_{2}}\varphi_{j}.
\end{align}
Then, by exploiting the expression of the last operator $\tilde{B}^{-1}$ and the formula  $(\ref{E2.21})$, we get the second point $(b)$ in $(iv)$.\\
The proof of theorem is closed.
\end{proof}
\begin{remark} In the case where $B_{1}=\{\varphi_{j}\}_{j\in \mathbb{Z}_{+}}$ is just only an orthonormal family and does not form a complete system in $\mathcal{H}_{1}=L^{2}(M_{1},d\mu_{1})$. The theorem $(\ref{T2.1})$ remains true if we take the closed subspace spanned by $\{\varphi_{j}\}_{j\in \mathbb{Z}_{+}}$ instead of the whole space $\mathcal{H}_{1}.$
\end{remark}
\begin{remark}
Notice that there exists a R.K.H.S which is not of $L^{2}-$type, as will be seen in the next section for the Dirichlet spaces \cite{Elf}. For these type-spaces, the theorem $(\ref{T2.1})$ can not be directly applied. However, to overcome this problem we can state the following theorem which is just a light modification of the first one.
\end{remark}
\begin{theorem}\label{T2.2} Let $L^{2}(M_{1},d\mu_{1})$ be a separable Hilbert space and let $\mathcal{A}$ be a reproducing kernel Hilbert space equipped with a scalar product $<,>_{\mathcal{A}}$. Let $\mathcal{B}_{1}=(\varphi_{j})_{j\in\mathbb{Z}^{+}}$ and $\mathcal{B}_{2}=(\psi_{j})_{j\in\mathbb{Z}^{+}}$ be an orthonormal basis of $L^{2}(M_{1},d\mu_{1})$ and $\mathcal{A}$  respectively. Then, we have the following statements
\begin{description}
  \item[i)] For each $z\in M_{2}$, the series
  \begin{align}\label{E2.28}
  K(z,w)=\displaystyle{\sum_{j\in\mathbb{Z}^{+}}}\overline{\varphi_{j}(w)}\psi_{j}(z),\hspace{0.25cm}w\in M_{1},
  \end{align}
  converges a.e$-d\mu_{1}$ and $K(z,.)\in L^{2}(M_{1},d\mu_{1}).$
  \item[ii)] The following integral transform
  \begin{align}\label{E2.29}
		\nonumber&L^{2}(M_{1},d\mu_{1})\longrightarrow \mathcal{A}\\
		&\varphi\longmapsto B[\varphi](z):=\int_{M_{1}}K(z,w)\varphi(w)d\mu_{1}(w),\hspace{0.25cm}z\in M_{2},
		\end{align}
defines a unitary isomorphism from $L^{2}(M_{1},d\mu_{1})$ into the R.K.H.S $\mathcal{A}.$
 \item[iii)] The inverse isomorphism $B^{-1}$ is given by the following series
		\begin{align}\label{E2.30}
		\nonumber&\mathcal{A} \longrightarrow L^{2}(M_{1},d\mu_{1})\\
		&\psi\longmapsto B^{-1}[\psi]:=\sum_{j=0}^{+\infty}<\psi,\psi_{j}>_{\mathcal{A}}\varphi_{j}.
		\end{align}
\end{description}
\end{theorem}
\begin{proof}
For $(i)$ and $(ii)$ the proofs given in theorem $(\ref{T2.1})$ remain unchanged.\\
For the proof of $(iii)$, we consider the associated Bargmann transform
\begin{align}\label{E2.31}
		\nonumber&L^{2}(M_{1},d\mu_{1})\longrightarrow (\mathcal{A},<,>_{\mathcal{A}})\\
		&\varphi\longmapsto B[\varphi]:=\sum_{j=0}^{+\infty}<\varphi,\varphi_{j}>_{\mathcal{H}_{1}}\psi_{j}.
		\end{align}
By a direct computation with the use of the Fourier's series theory in the Hilbert space, we can prove that the inverse of the transformation $B$ is given by the following Fourier series
\begin{align}\label{E2.32}
B^{-1}[\psi]=\sum_{j=0}^{+\infty}<\psi,\psi_{j}>_{\mathcal{A}}\varphi_{j}.
\end{align}
\end{proof}
Thanks to the above theorem, we can rederive some classical Bargmann transforms and its generalizations. Indeed, as examples we recover the classical Bargmann transform on $\mathbb{C}$ \cite{Bar}, the second Bargmann transform on the unit disk $\mathbb{D}$ \cite{Bar} and its generalization given by Intissar \emph{et} \emph{al.} \cite{Elw}.\\
Firstly, we are concerned with the classical Bargmann transform.
\subsection{Derivation of the classical Bargmann transform on the complex plan}
In this subsection, we deal with the following Bargmann transform
\begin{align}\label{E2.33}
		\nonumber&L^{2}(\mathbb{R},dx)\longrightarrow \emph{A}^{2}(\mathbb{C})\subset L^{2}(\mathbb{C},e^{-\mid z\mid^{2}}d\lambda(z))\\
		&\varphi\longmapsto B[\varphi](z):=\int_{-\infty}^{+\infty}\pi^{\frac{-3}{4}}\exp(-\frac{z^{2}}{2}+\sqrt{2}xz-\frac{x^{2}}{2})
\varphi(x)dx,\hspace{0.25cm}z\in \mathbb{C}.
\end{align}
The transformation $B$ is one to one isometry from $L^{2}(\mathbb{R},dx)$ into the Bargmann space  $\emph{A}^{2}(\mathbb{C})$ of holomorphic functions on $\mathbb{C}$ being $L^{2}-$integrable with respect to the Gaussian measure $d\mu_{2}(z)=e^{-\mid z\mid^{2}}d\lambda(z)$, where $d\lambda(z)$ is the Lebesgue measure on $\mathbb{C}.$\\
We recall that the Bargmann space $\emph{A}^{2}(\mathbb{C})$ is a R.K.H.S having
\begin{align}\label{E2.34}
\emph{K}(z,w)=\pi^{-1}e^{z\overline{w}},
\end{align}
as the reproducing kernel \cite{Bar}.\\
Now, in order to build this classical Bargmann transform by using theorem $(\ref{T2.1})$, we consider $$M_{1}=\mathbb{R},\hspace{0.25cm}d\mu_{1}(x)=e^{-x^{2}}dx,\hspace{0.25cm}M_{2}=\mathbb{C}\hspace{0.25cm}and\hspace{0.25cm}d\mu_{2}(z)=e^{-\mid z\mid^{2}}d\lambda(z).$$
It is well know that an orthogonal basis of $L^{2}(\mathbb{R},e^{-x^{2}}dx)$ is given in terms of the Hermite orthogonal polynomials \cite[p.250]{Mag},
\begin{align}\label{E2.35}
H_{j}(x)=j! \sum_{k=0}^{[\frac{j}{2}]}\frac{(-1)^{k}(2x)^{j-2k}}{k!(j-2k)!},
\end{align}
where $[\frac{j}{2}]$ is the integer part of the number $\frac{j}{2}.$\\
The corresponding normalized polynomials \cite[p.109]{Bea} defined by
\begin{align}\label{E2.36}
\varphi_{j}(x)=\frac{1}{\pi^{\frac{1}{4}}\sqrt{j!2^{j}}}H_{j}(x),\hspace{0.25cm}j\in\mathbb{Z}_{+},
\end{align}
forms an orthonormal basis of $L^{2}(\mathbb{R},e^{-x^{2}}dx)$.\\
 A canonical orthonormal basis \cite[p.201]{Bar} of the Bargmann space $\emph{A}^{2}(\mathbb{C})$ is given by
\begin{align}\label{E2.37}
\psi_{j}(z)=\frac{z^{j}}{\sqrt{\pi j!}}, \hspace{0.25cm}j\in \mathbb{Z}_{+}.
\end{align}
By using $(i)$ of theorem $(\ref{T2.1})$, the series
\begin{align}\label{E2.38}
  \nonumber K(z,x)&=\displaystyle{\sum_{j\in\mathbb{Z}^{+}}}\overline{\varphi_{j}(x)}\psi_{j}(z)\\
  &=\pi^{-\frac{3}{4}}\displaystyle{\sum_{j\in\mathbb{Z}^{+}}}\frac{H_{j}(x)}{j!}(\frac{z}{\sqrt{2}})^{j},
  \end{align}
converges, for each $z\in \mathbb{C},$ $\texttt{a.e}-d\mu_{1}(x)$ $in$ $\mathbb{R}.$\\
Using the generating  function \cite[p.102]{Mou} for the Hermite polynomials,
\begin{align}\label{E2.39}
\sum_{j=0}^{+\infty}\frac{H_{j}(x)}{j!}t^{j}=\exp(2xt-t^{2}),
\end{align}
for $t=\frac{z}{\sqrt{2}},$ the expression $(\ref{E2.38})$ becomes
\begin{align}\label{E2.40}
K(z,x)=\pi^{-\frac{3}{4}}\exp(\sqrt{2}xz-\frac{z^{2}}{2}).
\end{align}
Now, by using  $(ii)$ of theorem $(\ref{T2.1})$, the following transform
\begin{align}\label{E2.41}
\tilde{B}[\varphi](z):=\int_{-\infty}^{+\infty}\pi^{\frac{-3}{4}}\exp(-\frac{z^{2}}{2}+\sqrt{2}xz-x^{2})
\varphi(x)dx,\hspace{0.25cm}z\in \mathbb{C},
\end{align}
defines a unitary isomorphism from $L^{2}(\mathbb{R},e^{-x^{2}}dx)$
 into the Bargmann space $\emph{A}^{2}(\mathbb{C})$.\\
 It is easy to see that the following linear transform
 \begin{align}\label{E2.42}
		\nonumber T:\hspace{0.25cm}L^{2}(\mathbb{R},dx)&\longrightarrow L^{2}(\mathbb{R},e^{-x^{2}}dx)\\
		&f\longmapsto e^{\frac{x^{2}}{2}}f,
\end{align}
 is a unitary isomorphism. Then, by the composition $B=\tilde{B}\circ T$, we recover the Bargmann transform defined in $(\ref{E2.33}).$\\
 Moreover, by using $(iii)$ of theorem $(\ref{T2.1})$, the inverse isomorphism $B^{-1}$ is given by the integral transform
		\begin{align}\label{E2.43}
		\nonumber&L^{2}(\mathbb{C},e^{-\mid z\mid^{2}}d\lambda(z))\longrightarrow L^{2}(\mathbb{R},dx)\\
		&\psi\longmapsto B^{-1}[\psi](x):=\int_{\mathbb{C}}\pi^{\frac{-3}{4}}\exp(-\frac{\overline{z}^{2}}{2}+\sqrt{2}x\overline{z}-\frac{x^{2}}{2})
\psi(z)e^{-\mid z\mid^{2}}d\lambda(z),\hspace{0.25cm}x\in \mathbb{R}.
\end{align}
Secondly, we will be interested in the so-called second Bargmann transform \cite[p.203]{Bar}.
\subsection{Derivation of the second Bargmann transform on the unit disk}
In this subsection, we treat the following second Bargmann transform
\begin{align}\label{E2.44}
		\nonumber&L^{2}(\mathbb{R}_{+},\frac{x^{\delta}}{\Gamma(1+\delta)}dx)\longrightarrow  \emph{A}^{2,\delta}(\mathbb{D}) \subset L^{2}(\mathbb{D},\frac{\delta}{\pi}(1-\mid z\mid^{2})^{\delta-1}d\lambda(z))\\
		&\varphi\longmapsto B[\varphi](z):=\int_{0}^{+\infty}\frac{\exp(-\frac{x}{2}(\frac{1+z}{1-z}))}{(1-z)^{\delta+1}}
\varphi(x)\frac{x^{\delta}}{\Gamma(1+\delta)}dx,\hspace{0.25cm}z\in \mathbb{C},
\end{align}
where $\delta$ is a positive real number representing the parameter $\gamma$ in \cite[p.203]{Bar} and $\Gamma$ is the Euler Gamma function \cite[p.1]{Mag}. The transformation $B$ maps  isometrically the Hilbert space $L^{2}(\mathbb{R}_{+},\frac{x^{\delta}}{\Gamma(1+\delta)}dx)$ into the Bergman space  $\emph{A}^{2,\delta}(\mathbb{D})$ of holomorphic functions on unit disk $\mathbb{D}$ being $L^{2}-$integrable with respect to the  measure $d\mu_{2}(z)=\frac{\delta}{\pi}(1-\mid z\mid^{2})^{\delta-1}d\lambda(z)$.\\
To build the above transformation by the use of theorem $(\ref{T2.1})$, we first recall that the
Bergman space $\emph{A}^{2,\delta}(\mathbb{D})$ is a R.K.H.S with the following reproducing kernel \cite{Hed,Zhu2}.
\begin{align}\label{E2.45}
\emph{K}^{\delta}(z,w)=\frac{1}{(1-z\overline{w})^{\delta+1}}.
\end{align}
Secondly, we consider
$$M_{1}=\mathbb{R}_{+},\hspace{0.25cm}d\mu_{1}(x)=x^{\delta}e^{-x}dx,\hspace{0.25cm}M_{2}=\mathbb{D}\hspace{0.25cm}
and\hspace{0.25cm}d\mu_{2}(z)=\frac{\delta}{\pi}(1-\mid z\mid^{2})^{\delta-1}d\lambda(z).$$
The classical orthogonal basis of  $L^{2}(\mathbb{R}_{+},e^{-x}x^{\delta} dx)$ is given by the Laguerre polynomial
\cite[p.113]{Bea}
\begin{align}\label{E2.46}
L_{j}^{(\delta)}(x)=\sum_{k=0}^{j}\frac{(-1)^{k}(\delta+1)_{k}}{k!(j-k)!}x^{k},
\end{align}
where $(\alpha)_{k}=\alpha(\alpha+1)...(\alpha+k-1)$ is the Bokhammer symbol.\\
The corresponding orthonormalized polynomials \cite[p.115]{Bea} defined by
\begin{align}\label{E2.47}
\varphi_{j}(x)=\sqrt{\frac{j!}{\Gamma(\delta+j+1)}}L_{j}^{(\delta)}(x),\hspace{0.25cm}j\in\mathbb{Z}_{+},
\end{align}
constitute an orthonormal basis of the Hilbert space $L^{2}(\mathbb{R}_{+},x^{\delta}e^{-x} dx)$.\\
It is not hard to see that the family
\begin{align}\label{E2.48}
\psi_{j}(z)=\sqrt{\frac{\Gamma(1+\delta+j)}{j!\delta\Gamma(\delta)}}z^{j}, \hspace{0.25cm}j\in \mathbb{Z}_{+},
\end{align}
forms an orthonormal basis of the Bergman space $\emph{A}^{2,\delta}(\mathbb{D})$.\\
Now, by using  $(i)$ of theorem $(\ref{T2.1})$, we can see that the following series
\begin{align}\label{E2.49}
  \nonumber K(z,x)&=\displaystyle{\sum_{j\in\mathbb{Z}^{+}}}\overline{\varphi_{j}(x)}\psi_{j}(z)\\
  &=[\delta\Gamma(\delta)]^{-\frac{1}{2}}\displaystyle{\sum_{j\in\mathbb{Z}^{+}}}z^{j}L_{j}^{(\delta)}(x),
  \end{align}
converges for each fixed $z\in \mathbb{D},$ $\texttt{a.e}-d\mu_{1}(x)$ $in$ $\mathbb{R}_{+}.$\\
By applying the canonical functional relation $\Gamma(\delta+1)=\delta\Gamma(\delta)$, the formula (\ref{E2.49}) can be rewritten as
\begin{align}\label{E2.50}
  K(z,x)&=[\Gamma(\delta+1)]^{-\frac{1}{2}}\displaystyle{\sum_{j\in\mathbb{Z}^{+}}}z^{j}L_{j}^{(\delta)}(x).
  \end{align}
Now, with the help of the generating  function \cite[p.114]{Bea} for the Laguerre  polynomials,
\begin{align}\label{E2.51}
\sum_{j=0}^{+\infty}z^{j}L_{j}^{(\delta)}(x)=\frac{1}{(1-z)^{\delta+1}}\exp(-\frac{xz}{1-z}),
\end{align}
the function $K(z,x)$ given in equation $(\ref{E2.49})$ becomes
\begin{align}\label{E2.52}
K(z,x)=\frac{[\Gamma(\delta+1)]^{-\frac{1}{2}}}{(1-z)^{\delta+1}}\exp(-\frac{xz}{1-z}).
\end{align}
Its follows from $(ii)$ of theorem $(\ref{T2.1})$ that the integral transform
\begin{align}\label{E2.53}
\tilde{B}[\varphi](z):=\int_{0}^{+\infty}\frac{[\Gamma(\delta+1)]^{-\frac{1}{2}}}{(1-z)^{\delta+1}}
\exp(-\frac{xz}{1-z})\varphi(x)x^{\delta}e^{-x}dx,\hspace{0.25cm}z\in \mathbb{C},
\end{align}
maps isometrically the Hilbert space $L^{2}(\mathbb{R}_{+},x^{\delta}e^{-x} dx)$
 into the Bergman space $\emph{A}^{2,\delta}(\mathbb{C})$.\\
 Also in a natural way, we have the following unitary isomorphism
 \begin{align}\label{E2.54}
		\nonumber T:\hspace{0.25cm}L^{2}(\mathbb{R}_{+},&\frac{x^{\delta}}{\Gamma(1+\delta)}dx)\longrightarrow L^{2}(\mathbb{R}_{+},x^{\delta}e^{-x} dx)\\
		&\hspace{1cm}f\longmapsto \frac{e^{\frac{x}{2}}}{\sqrt{\Gamma(1+\delta)}}f.
\end{align}
Then, the  composition $B=\tilde{B}\circ T$ gives the  desired second  Bargmann transform defined in $(\ref{E2.44}).$\\
  Moreover, by using $(iii)$ of theorem $(\ref{T2.1})$ and the well known algebraic relation $B^{-1}=T^{-1}\circ \tilde{B}^{-1}$ the inverse isomorphism $B^{-1}$ can be given by the following integral transform
  \begin{align}\label{E2.55}
		\nonumber &\emph{A}^{2,\delta}(\mathbb{D}) \subset L^{2}(\mathbb{D},\frac{\delta}{\pi}(1-\mid z\mid^{2})^{\delta-1}d\lambda(z))\longrightarrow L^{2}(\mathbb{R}_{+},\frac{x^{\delta}}{\Gamma(1+\delta)}dx)\\
\nonumber\\
		&\psi\longmapsto B^{-1}[\psi](x):=\frac{\delta}{\pi}\int_{\mathbb{D}}
\frac{\exp(-\frac{x}{2}(\frac{1+\overline{z}}{1-\overline{z}}))}{(1-\overline{z})^{\delta+1}}
\psi(z)(1-\mid z\mid^{2})^{\delta-1}d\lambda(z),\hspace{0.25cm}x\in \mathbb{R}_{+}.
\end{align}
Now, we will be concerned with a new isometry called the generalized second Bargmann transform associated with the hyperbolic Landau-Levels on the Poincaré disk \cite{Elw}.
\subsection{Derivation of the generalized second Bargmann transform}
Our goal in this subsection is to show, concretely, that the generalized second Bargmann transform \cite{Elw} can be built rigourously by applying theorem $(\ref{T2.1})$ without using the coherent states language.\\
  Precisely, we will be concerned with the $L^{2}-$eigenspaces
 \begin{align}\label{E2.56}
 A_{\ell}^{2,\nu}(\mathbb{D}):=\{F\in L^{2,\nu}(\mathbb{D}),\hspace{0.25cm}H_{\nu}F=\varepsilon^{\nu}_{\ell}F\},
 \end{align}
 where $\varepsilon^{\nu}_{\ell}=4\ell(2\nu-\ell-1)$, $\nu>\frac{1}{2}$ and $H_{\nu}$ is the second order differential operator
 \begin{align}\label{E2.57}
 H_{\nu}=-4(1-\mid z\mid^{2})[(1-\mid z\mid^{2})\frac{\partial^{2}}{\partial z\partial\overline{z}}-2\nu\overline{z}\frac{\partial}{\partial\overline{z}}],
 \end{align}
 acting on the Hilbert space $L^{2,\nu}(\mathbb{D}):=(L^{2}(\mathbb{D}),(1-\mid z\mid^{2})^{2\nu-2}d\lambda(z))$ with a maximal domain
 \begin{align}\label{E2.58}
 D(H_{\nu}):=\{F\in L^{2,\nu}(\mathbb{D}),\hspace{0.25cm}H_{\nu}F\in L^{2,\nu}(\mathbb{D})\}.
 \end{align}
 More precisions on the origin of the expression of the differential operator $H_{\nu}$, as well as some area in which has been considered, will be given in the next section.\\
  The operator $H_{\nu}$ is self adjoint and the associated point spectrum is given by
\begin{align}\label{E2.59}
 \sigma_{p}(H_{\nu})=\{4\ell(2\nu-\ell-1),\hspace{0.25cm}\ell=0,1,2,...,[\nu-\frac{1}{2}]\},
 \end{align}
where $[x]$ is the integer part of real number $x.$\\
According to \cite{Elw,Hep,Zha}, the reproducing kernel of the Hilbert space $A_{\ell}^{2,\nu}(\mathbb{D})$ is
 \begin{align}\label{E2.60}
 \nonumber K_{\ell}^{\nu}(z,w)&=(\frac{2(\nu-\ell)-1}{\pi})(1-z\overline{w})^{-2\nu}(\frac{\mid1-z\overline{w}\mid^{2}}{(1-\mid z\mid^{2})(1-\mid w\mid^{2})})^{\ell}\\
 &\times P_{\ell}^{(0,2(\nu-\ell)-1)}(2\frac{(1-\mid z\mid^{2})(1-\mid w\mid^{2})}{\mid1-z\overline{w}\mid^{2}}-1).
\end{align}
The generalized second Bargmann transform reads as
\begin{align}\label{E2.61}
 \nonumber B_{\ell}^{\nu}:\hspace{0.25cm}L^{2}(\mathbb{R}^{*}_{+},& \frac{dx}{x})\longrightarrow A_{\ell}^{2,\nu}(\mathbb{D})\\
 &\varphi\longmapsto B_{\ell}^{\nu}[\varphi],
 \end{align}
 where
 \begin{align}\label{E2.62}
 \nonumber\hspace{2cm}B_{\ell}^{\nu}[\varphi](z):&=\bigg(\frac{\ell!(2(\nu-\ell)-1)}{\pi\Gamma(2\nu-\ell)}\bigg)^{\frac{1}{2}}
 \\
 &\hspace{-4cm}\times\int_{0}^{+\infty}x^{\nu-\ell}\bigg(\frac{1-\mid z\mid^{2}}{\mid1-z\mid^{2}}\bigg)^{-\ell}(1-z)^{-2\nu}\exp(-\frac{x}{2}(\frac{z+1}{z-1}))L^{2(\nu-\ell)-1}_{\ell}\bigg(x\frac{1-\mid z\mid^{2}}{\mid1-z\mid^{2}}\bigg)\varphi(x)\frac{dx}{x}.
 \end{align}
 In order to recover the above isometry, we consider
$$M_{1}=\mathbb{R}_{+},\hspace{0.25cm}d\mu_{1}(x)=x^{\alpha}e^{-x}dx,\hspace{0.2cm}\alpha>-1,
\hspace{0.25cm}M_{2}=\mathbb{D}\hspace{0.25cm}and\hspace{0.25cm}d\mu_{2}(z)=(1-\mid z\mid^{2})^{2\nu-2}d\lambda(z).$$
  Using (\ref{E2.47}), an orthonormal basis of the Hilbert space $L^{2}(\mathbb{R}_{+},x^{\alpha}e^{-x}dx)$ can be given in terms of the Laguerre polynomial by
 \begin{align}\label{E2.63}
\varphi_{j}^{(\alpha)}(x)=\sqrt{\frac{j!}{\Gamma(\alpha+j+1)}}L_{j}^{(\alpha)}(x),\hspace{0.25cm}j\in\mathbb{Z}_{+}.
\end{align}
An orthonormal basis \cite{Elw} of the RKHS  $A_{\ell}^{2,\nu}(\mathbb{D})$ can be expressed in terms of the Jacobi polynomials $P_{j}^{(\alpha,\beta)}(.)$  (\cite[p.116]{Bea}) as follows
\begin{align}\label{E2.64}
\nonumber \psi_{j}^{\nu,\ell}(z)&=(-1)^{j}\bigg(\frac{2(\nu-\ell)-1}{\pi}\bigg)^{\frac{1}{2}}\bigg(\frac{j!\Gamma(2(\nu-\ell)+\ell)}{\ell!\Gamma(2(\nu-\ell)+j)}\bigg
)^{\frac{1}{2}}
(1-\mid z\mid^{2})^{-\ell}(\overline{z})^{\ell-j}\\
&\times P_{j}^{(\ell-j,2(\nu-\ell)-1)}(1-2\mid z\mid^{2}),\hspace{0.25cm}j\in\mathbb{Z}_{+}.
\end{align}
From $(i)$ of theorem $(\ref{T2.1})$, the following series
\begin{align}\label{E2.65}
  \nonumber K(z,x)&=\displaystyle{\sum_{j\in\mathbb{Z}^{+}}}\overline{\varphi_{j}^{\alpha}(x)}\psi_{j}^{\nu,\ell}(z)\\
  \nonumber&=\bigg(\frac{2(\nu-\ell)-1}{\pi}\bigg)^{\frac{1}{2}}\bigg(\frac{\Gamma(2(\nu-\ell)+\ell)}{\ell!}\bigg)^{\frac{1}{2}}
  (1-\mid z\mid^{2})^{-\ell}(\overline{z})^{\ell}\\
  &\times\displaystyle{\sum_{j\in\mathbb{Z}^{+}}}
  \frac{j!(-\overline{z}^{-1})^{j}}{[\Gamma(1+j+\alpha)\Gamma(2(\nu-\ell)+j)]\frac{1}{2}}L_{j}^{(\alpha)}(x)
  P_{j}^{(\ell-j,2(\nu-\ell)-1)}(1-2\mid z\mid^{2}),
\end{align}
converges for each $z\in \mathbb{D}$ and $a.e-d\mu_{1}(x)$.\\
By using together the symmetry relation \cite[p.210]{Mag}
\begin{align}\label{E2.66}
P_{j}^{(\delta,\beta)}(x)=(-1)^{j}P_{j}^{(\beta,\delta)}(-x),
\end{align}
 for $\delta=\ell-j$, $\beta=2(\nu-\ell)-1$, $x=1-2\mid z\mid^{2}$ and the formula \cite[p.212]{Mag}
 \begin{align}\label{E2.67}
P_{j}^{(\beta,\gamma)}(t)=\tbinom{j+\beta}{j}\prescript{}{2}{F}_1^{}(-j,\beta+\gamma+j+1,1+\beta,\frac{1-t}{2}),
\end{align}
for $\gamma=\ell-j$, $\beta=2(\nu-\ell)-1$ and $t=2\mid z\mid^{2}-1$, the equation $(\ref{E2.65})$ becomes
\begin{align}\label{E2.68}
  \nonumber &K(z,x)=\bigg(\frac{2(\nu-\ell)-1}{\pi}\bigg)^{\frac{1}{2}}\bigg(\frac{\Gamma(2(\nu-\ell)+\ell)}{\ell!}\bigg)^{\frac{1}{2}}
  (1-\mid z\mid^{2})^{-\ell}(\overline{z})^{\ell}\\
  &\times\displaystyle{\sum_{j\in\mathbb{Z}^{+}}}
  \frac{j!(-\overline{z}^{-1})^{j}(-1)^{j}}{[\Gamma(1+j+\alpha)\Gamma(2(\nu-\ell)+j)]\frac{1}{2}}
  \tbinom{2(\nu-\ell)-1+j}{j}L_{j}^{(\alpha)}(x)\prescript{}{2}{F}_1^{}(-j,2(\nu-\ell)+\ell,2(\nu-\ell),1-\mid z\mid^{2}).
\end{align}
In order to use the bilateral generating function \cite[p.213]{Ran}
\begin{align}\label{E2.69}
\nonumber \sum_{j=0}^{+\infty}(\lambda)^j\prescript{}{2}{F}_1^{}(-j,b,1+\alpha,y)L_j^{(\alpha)}(x)&=
\frac{(1-\lambda)^{b-1-\alpha}}{(1-\lambda+\lambda y)^{b}} \exp (\frac{-x\lambda}{1-\lambda})\\
&\times_1 F_1(b,1+\alpha,\frac{xy\lambda}{(1-\lambda)(1-\lambda+\lambda y)}),
\end{align}
we chose $\alpha=2(\nu-\ell)-1$ in $(\ref{E2.68})$ which is positive (since $\nu>\frac{1}{2}$) and we set $\lambda=\frac{1}{\overline{z}}$, $b=2(\nu-\ell)+\ell$ and $y=1-\mid z\mid^{2}$. Then, after replacing  $\binom{2(\nu-\ell)-1+j}{j}$ by $\frac{\Gamma(2(\nu-\ell)+j)}{j!\Gamma(2(\nu-\ell))}$ and using the above generating formula, the involved series in the equation $(\ref{E2.68})$ becomes
\begin{align}\label{E2.70}
\nonumber&\frac{1}{\Gamma(2(\nu-\ell))}\sum_{j=0}^{+\infty}(\frac{1}{\overline{z}})^j\prescript{}{2}{F}_1^{}(-j,2(\nu-\ell)+\ell,2(\nu-\ell),1-|z|^2)L_j^{2(\nu-\ell)-1}(x)\\
&=\frac{1}{\Gamma(2(\nu-\ell))}\frac{(\overline{z}-1)^{\ell} \overline{z}^{-\ell}}{(1-z)^{2(\nu-\ell)+\ell}} \exp (\frac{-x}{\overline{z}-1})_1 F_1(2(\nu-\ell)+\ell,2(\nu-\ell),\frac{-x(1-|z|^2)}{|1-z|^2}).
\end{align}
Applying the formula \cite[p.267]{Mag}
\begin{align}\label{E2.71}
\prescript{}{1}{F}_1^{}(a,c,t)=e^{t}\prescript{}{1}{F}_1^{}(c-a,c,-t),
\end{align}
for $a=2(\nu-\ell)+\ell$, $c=2(\nu-\ell)$ and $t=\frac{-x(1-\mid z\mid^{2})}{\mid z-1\mid^{2}}$, the equation $(\ref{E2.70})$ takes the form
\begin{align}\label{E2.72}
\nonumber&\frac{1}{\Gamma(2(\nu-\ell))}\sum_{j=0}^{+\infty}(\frac{1}{\overline{z}})^j\prescript{}{2}{F}_1^{}(-j,2(\nu-\ell)+\ell,2(\nu-\ell),1-|z|^2)L_j^{2(\nu-\ell)-1}(x)\\
&=\frac{1}{\Gamma(2(\nu-\ell))}\frac{(\overline{z}-1)^{\ell} \overline{z}^{-\ell}}{(1-z)^{2(\nu-\ell)+\ell}} \exp (\frac{-x}{\overline{z}-1})\exp(\frac{-x(1-|z|^2)}{|1-z|^2})\prescript{}{1}{F}_1^{}(-\ell,2(\nu-\ell),\frac{x(1-|z|^2)}{|1-z|^2}).
\end{align}
Using a direct computation and the following relation \cite[p.287]{Mag}
\begin{align}\label{E2.73}
\prescript{}{1}{F}_1^{}(-n,1+\gamma,s)=\frac{n!\Gamma(1+\gamma)}{\Gamma(n+\gamma +1)}L_{n}^{(\gamma)}(s),
\end{align}
for $n=\ell$, $\gamma=2(\nu-\ell)-1$ and  $s=\frac{x(1-|z|^2)}{|1-z|^2}$, the function $k(z,x)$ defined in $(\ref{E2.68})$ becomes
\begin{align}\label{E2.74}
k(z,x)=\bigg(\frac{\ell!(2(\nu-\ell)-1)}{\pi\Gamma(2\nu-\ell)}\bigg)
^{\frac{1}{2}}(\frac{1-|z|^2}{|1-z|^2})^{-\ell}(1-z)^{-2\nu}\exp(\frac{xz}{z-1})L_{\ell}^{2(\nu-\ell)-1}(\frac{x(1-|z|^2)}{|1-z|^2}).
\end{align}
Then, by using $(ii)$ of theorem $(\ref{T2.1})$, we obtain the following isometry
\begin{align}\label{E2.75}
 \nonumber\tilde{B}_{\ell}^{\nu}:\hspace{0.25cm}L^{2}(\mathbb{R}^{*}_{+},& x^{2(\nu-\ell)-1}e^{-x}dx)\longrightarrow A_{\ell}^{2,\nu}(\mathbb{D})\\
 &\varphi\longmapsto \tilde{B}_{\ell}^{\nu}[\varphi],
 \end{align}
 where
 \begin{align}\label{E2.76}
 \tilde{B}_{\ell}^{\nu}[\varphi](z):&=\bigg(\frac{\ell!(2(\nu-\ell)-1)}{\pi\Gamma(2\nu-\ell)}\bigg)^{\frac{1}{2}}
 (\frac{1-\mid z\mid^{2}}{\mid1-z\mid^{2}})^{-\ell}(1-z)^{-2\nu}\\
 \nonumber&\times\int_{0}^{+\infty}\exp(\frac{xz}{z-1})L^{2(\nu-\ell)-1}_{\ell}(x\frac{1-\mid z\mid^{2}}{\mid1-z\mid^{2}})x^{2(\nu-\ell)-1}\exp(-x)\varphi(x)dx.
 \end{align}
Now, we consider the following canonical isometry
\begin{align}\label{E2.77}
 \nonumber T:\hspace{0.25cm}L^{2}(\mathbb{R}^{*}_{+},& \frac{dx}{x})\longrightarrow L^{2}(\mathbb{R}^{*}_{+}, x^{2(\nu-\ell)-1}e^{-x}dx)\\
 &\varphi\longmapsto x^{\ell-\nu}e^{\frac{x}{2}}\varphi.
 \end{align}
The composition $T\circ \tilde{B}_{\ell}^{\nu}$ gives our desired isometry $B_{\ell}^{\nu}$ defined in $(\ref{E2.62})$.\\
Moreover, by using $(iii)$ of theorem $(\ref{T2.1})$ the inverse isomorphism\\ $(B_{\ell}^{\nu})^{-1}=(\tilde{B}_{\ell}^{\nu})^{-1}\circ T^{-1}$ is given by the following integral transform
\begin{align}\label{E2.78}
A_{\ell}^{2,\nu}(\mathbb{D})&\longrightarrow L^{2}(\mathbb{R}^{*}_{+},\frac{dx}{x})\\
 \nonumber&\psi\longmapsto (B_{\ell}^{\nu})^{-1}[\psi],
 \end{align}
 where
 \begin{align}\label{E2.79}
 \nonumber(B_{\ell}^{\nu})^{-1}[\psi](x):=\bigg(\frac{\ell!(2(\nu-\ell)-1)}{\pi\Gamma(2\nu-\ell)}\bigg)^{\frac{1}{2}}
 \\
 &\hspace{-4cm}\hspace{-4cm}\times\int_{\mathbb{D}}x^{\nu-\ell}\bigg(\frac{1-\mid z\mid^{2}}{\mid1-z\mid^{2}}\bigg)^{-\ell}(1-\overline{z})^{-2\nu}
 \exp(-\frac{x}{2}(\frac{\overline{z}+1}{\overline{z}-1}))L^{2(\nu-\ell)-1}_{\ell}\bigg(x\frac{1-\mid z\mid^{2}}{\mid1-z\mid^{2}}\bigg)\psi(z)(1-\mid z\mid^{2})^{2\nu-2}d\lambda(z).
 \end{align}
\section{Dirichlet space and associated Bargmann transform}
\subsection{Dirichlet space as harmonic space}
In this subsection, we discuss the Dirichlet space on the unit disk. Before going ahead, we recall that the classical Dirichlet space $\mathcal{D}$ on the unit disk $\mathbb{D}$ is defined as the class of analytic functions
\begin{align}\label{E3.1}
f(z)=\sum_{n=0}^{+\infty}a_{n}z^{n},
\end{align}
for which the Dirichlet integral
\begin{align}\label{E3.2}
\mathbf{D}(f)=\frac{1}{\pi}\int_{\mathbb{D}}\mid f^{'}(z)\mid^{2}d\lambda(z)=\sum_{n=0}^{+\infty}n\mid a_{n}\mid^{2},
\end{align}
is finite \cite{Elf}.\\
The Dirichlet integral appears as the trace of the commutator of the Toeplitz operator $T_{f}$ and its adjoint $T^{*}_{f}$ in the case where $f$ is a bounded analytic function on $\mathbb{D}$ (\cite{Ara1}).\\
 The Dirichlet space intervenes with a great importance in operator theory, as well as in modern analysis and potential theory (see \cite{Elf} and the references therein).\\
The Dirichlet space $\mathcal{D}$ is characterized as the unique Hilbert space of analytic functions on $\mathbb{D}$ that are invariant by the M\"{o}bius group transforms \cite{Ara2}
\begin{align}\label{E3.3}
\varphi_{(a,b)}:=b\frac{z-a}{1-\overline{a}z}, \hspace{0.25cm}\mid a\mid<1, \hspace{0.25cm}\mid b\mid=1.
\end{align}
Also, in \cite{Wu}, the Dirichlet space  $\mathcal{D}$ can be viewed as a subspace of the Sobolev space $\mathscr{L}^{2,1}(\mathbb{D})$ defined as the completion of the space of the functions of class one on $\mathbb{D}$ under the norm
\begin{align}\label{E3.4}
\parallel f\parallel=\{\mid\int_{\mathbb{D}}fdA\mid^{2}+\int_{\mathbb{D}}(\mid\frac{\partial f}{\partial z}\mid+\mid\frac{\partial f}{\partial \overline{z}}\mid)dA\}^{\frac{1}{2}},
\end{align}
where $dA(z)=\frac{1}{\pi} d\lambda(z)$ and $\frac{\partial }{\partial z}$, $\frac{\partial }{\partial \overline{z}}$ are the classical Cauchy operators.\\
Here, our first goal is to give a new characterization for the  Dirichlet space $\mathcal{D}$.\\
Concretely, we shall prove that the Dirichlet space  $\mathcal{D}$ is the null space of a suitable partial differential operator.\\
The construction of the desired operator needs some harmonic analysis tools. Precisely, let us recall that the M\"{o}bius group $G=SU(1,1)$ consists of all $2\times2$ complex matrices
\begin{align}\label{E3.5}
g=\left(
    \begin{array}{cc}
      \alpha & \beta \\\\
      \overline{\beta} & \overline{\alpha} \\
    \end{array}
  \right),\hspace{0.25cm} \mid\alpha\mid^{2}-\mid\beta\mid^{2}=1.
\end{align}
It acts on $\mathbb{D}$ by the homographical maps
\begin{align}\label{E3.6}
z\longmapsto g.z=g(z)=\frac{\alpha z+\beta}{\overline{\beta}z+\overline{\alpha}}, \hspace{0.25cm} z\in \mathbb{D}.
\end{align}
It is noted that all holomorphic automorphisms on $\mathbb{D}$ can be obtained by such a symmetry. The Lie algebra  $su(1,1)$ of the group $SU(1,1)$ is generated by the basis
\begin{align}\label{E3.7}
Z=\left(
    \begin{array}{cc}
      i & 0 \\
      0 & -i \\
    \end{array}
  \right),\hspace{0.25cm}
A=\left(
    \begin{array}{cc}
      0 & 1 \\
      1 & 0\\
    \end{array}
  \right),\hspace{0.25cm}
  B=\left(
    \begin{array}{cc}
      0 & i \\
      -i & 0\\
    \end{array}
  \right).
\end{align}
The associated Casimir element ( see \cite{Far} for the general theory) is given by \cite{Hep}
\begin{align}\label{E3.8}
\Box=Z^{2}-A^{2}-B^{2}.
\end{align}
Let $\gamma$ be a fixed real parameter. Consider the Hilbert space $L^{2}(\mathbb{D},d\mu_{\gamma})$ where
$d\mu_{\gamma}(z)=(1-\mid z\mid^{2})^{\gamma-2}d\lambda(z)$ and  $d\lambda(z)=dxdy$ is the Lebesgue measure on $\mathbb{D}$.\\
For $g\in SU(1,1)$, we consider the following transformation
\begin{align}\label{E3.9}
T^{\gamma}(g):\hspace{0.25cm}f(z)\longmapsto f(g(z))\{g^{'}(z)\}^{\frac{\gamma}{2}}=(\overline{\beta}z+\overline{\alpha})^{-\gamma}f(\frac{\alpha z+\beta}{\overline{\beta}z+\overline{\alpha}}).
\end{align}
In the case where $\gamma\in\mathbb{R}$$\backslash$$\mathbb{Z}$, $T^{\gamma}$ gives a projective representation of the group $SU(1,1)$ and a genuine representation of the universal covering group of $SU(1,1)$. In the case $\gamma\in \mathbb{Z}$, $T^{\gamma}$ gives a continuous unitary representation of the group $SU(1,1)$. The transformation $T^{\gamma}$ induces a representation of the Lie algebra $su(1,1)$ and its universal enveloping algebra on the space of $C^{\infty}-$vectors for $T^{\gamma}$, which is also denoted by $T^{\gamma}$. By a direct computation, the corresponding vectors fields are
 \begin{align}\label{E3.10}
 T^{\gamma}(Z)=2iz\frac{\partial}{\partial z}-2i\overline{z}\frac{\partial}{\partial \overline{z}}+i\gamma,
 \end{align}
\begin{align}\label{E3.11}
 T^{\gamma}(A)=(1-z^{2})\frac{\partial}{\partial z}+(1-\overline{z}^{2})\frac{\partial}{\partial \overline{z}}-\gamma z,
 \end{align}
 \begin{align}\label{E3.12}
 T^{\gamma}(B)=i(1+z^{2})\frac{\partial}{\partial z}-i(1+\overline{z}^{2})\frac{\partial}{\partial \overline{z}}+i\gamma\overline{z}.
 \end{align}
Therefore, the Casimir operator (or the invariant Laplacian) is given by
\begin{align}\label{E3.13}
\square_{\gamma}=T^{\gamma}(\square)=-4(1-\mid z\mid^{2})\{(1-\mid z\mid^{2})\frac{\partial^{2}}{\partial z\partial \overline{z}}
-\gamma\overline{z}\frac{\partial}{\partial\overline{z}}\}-\gamma^{2}+2\gamma.
\end{align}
This class of operators has been extensively considered by several authors in many subjects. In one dimension case with using the same class of such operators expressed in complex coordinate of the upper half plan, Fay \cite{Fay} has studied the spectral theory of these operators on the class of authomorphic functions associated with a Fuchsian group. The classical $L^{2}-$theory follows for the trivial group $G=\{e\}$ where $e$ is the identity element of $G$. In higher dimensions, G. Zhang \cite{Zha} has studied in an harmonic analysis framework the irreducible decomposition of a suitable representation of the group $SU(1,n)$. He determined the discrete part which is associated with the eigenspace of such operator. Also, he found the reproducing kernels of these eigenspaces.\\
Here, the above operators given in $(\ref{E3.13})$ will play an important role in the characterization of the Dirichlet space as harmonic space of a second order  elliptic partial differential operator. Concretely, we well deal with the operator $\square_{2}$. We have the following proposition.
\begin{proposition} \label{P3.1} Let $\tilde{\Delta}$ be the partial differential operator defined by
\begin{align}\label{E3.14}
 \tilde{\Delta}:=\square_{2}=-4(1-\mid z\mid^{2})[(1-\mid z\mid^{2})\frac{\partial^{2}}{\partial z\partial\overline{z}}-2\overline{z}\frac{\partial}{\partial\overline{z}}],
 \end{align}
 acting on the Hilbert space $L^{2}(\mathbb{D},d\lambda(z))$, with the dense domain
 \begin{align}\label{E3.15}
 D(\tilde{\Delta}):=\{F\in L^{2}(\mathbb{D},d\lambda(z)),\hspace{0.25cm}\tilde{\Delta}F\in L^{2}(\mathbb{D},d\lambda(z))\hspace{0.25cm}and\hspace{0.25cm}\frac{\partial F}{\partial z}\in L^{2}(\mathbb{D},d\lambda(z))\}.
 \end{align}
 Then, $\mathcal{D}$ is the null space of $\tilde{\Delta}$
 \begin{align}\label{E3.16}
 \mathcal{D}=\{F\in D(\tilde{\Delta}),\hspace{0.25cm}\tilde{\Delta}F=0\}.
 \end{align}
 \end{proposition}
\begin{proof}Let $f(z)=\sum_{j=0}^{+\infty}a_{j}z^{j}$ be a holomorphic function on the unit disk. By using the polar coordinates, we have
\begin{align}\label{E3.17}
\int_{\mathbb{D}}\mid f(z)\mid^{2}d\lambda(z)=2\pi \int_{0}^{1}\bigg[\int_{0}^{2\pi}\mid \sum_{j=0}^{+\infty} a_{j} r^{j}e^{ij\theta}\mid^{2}\frac{d\theta}{2\pi}\bigg]rdr.
\end{align}
By Parseval's formula, for each $r\in[0,1]$, we obtain
\begin{align}\label{E3.18}
\int_{0}^{2\pi}\mid \sum_{j=0}^{+\infty} a_{j} r^{j}e^{ij\theta}\mid^{2}\frac{d\theta}{2\pi}=\sum_{j=0}^{+\infty}r^{2j}\mid a_{j}\mid^{2}.
\end{align}
Inserting the last equality in the right hand side of the equation $(\ref{E3.17})$ and using the monotone convergence theorem, the equation $(\ref{E3.17})$ becomes
\begin{align}\label{E3.19}
\int_{\mathbb{D}}\mid f(z)\mid^{2}d\lambda(z)=\pi\sum_{j=0}^{+\infty}\frac{\mid a_{j}\mid^{2}}{j+1}.
\end{align}
Using the following inequality
\begin{align}\label{E3.20}
\pi\sum_{j=1}^{+\infty}\frac{\mid a_{j}\mid^{2}}{j+1}\leq\pi\sum_{j=1}^{+\infty}j\mid a_{j}\mid^{2},
\end{align}
we obtain the inclusion
\begin{align}\label{E3.21}
\mathcal{D}\subset\{f\hspace{0.2cm} \mathbb{D}\longmapsto \mathbb{C}\hspace{0.2cm}holomorphic,\hspace{0.25cm}\int_{\mathbb{D}}\mid f(z)\mid^{2}d\lambda(z)<\infty\}.
\end{align}
Recall that the Bergman space on the unit disk
\begin{align}\label{E3.22}
\mathcal{A}^{2}(\mathbb{D})=\{f\hspace{0.2cm} \mathbb{D}\longmapsto \mathbb{C}\hspace{0.2cm}holomorphic,\hspace{0.25cm}\int_{\mathbb{D}}\mid f(z)\mid^{2}d\lambda(z)<\infty\},
\end{align}
is a R.K.H.S with the following reproducing kernel
\begin{align}\label{E3.23}
K(z,w)=\frac{1}{\pi}(\frac{1}{1-z\overline{w}})^{2}.
\end{align}
Two references, in this context, are \cite{Hed,Zhu2}.\\
Now, we return back to the eigenspace  $\mathcal{A}^{2,\nu}_{\ell}(\mathbb{D})$ defined in the equation (\ref{E2.56}) in which we set $\nu=1$ and $\ell=0$. From the equation (\ref{E2.60}), it is not difficult to see that the reproducing kernel of the space $\mathcal{A}^{2,1}_{0}(\mathbb{D})$ is
\begin{align}\label{E3.24}
K^{1}_{0}(z,w)=\frac{1}{\pi}(\frac{1}{1-z\overline{w}})^{2}.
\end{align}
The both spaces $\mathcal{A}^{2}(\mathbb{D})$ and $\mathcal{A}^{2,1}_{0}(\mathbb{D})$ have the same reproducing kernel. Using the proposition $3.3$ in $\cite[p.9]{Pau}$, we get
\begin{align}\label{E3.25}
\mathcal{A}^{2}(\mathbb{D})=\mathcal{A}^{2,1}_{0}(\mathbb{D}):=\{F\in L^{2}(\mathbb{D},d\lambda(z)),\hspace{0.2cm}H_{1}F=0\},
\end{align}
where $H_{1}$ is the operator defined by $(\ref{E2.56})$ and $(\ref{E2.57})$ in which we have set $\nu=1$.\\
Then, the inclusion $(\ref{E3.21})$ combined with the equalities $(\ref{E3.22})$, $(\ref{E3.25})$ and the fact $\tilde{\Delta}F=H_{1}F$ lead to the following inclusion
\begin{align}\label{E3.26}
\hspace{-0.1cm}\mathcal{D}\subset\{F\in L^{2}(\mathbb{D},d\lambda(z)),\hspace{0.1cm}\frac{\partial F}{\partial z}\in L^{2}(\mathbb{D},d\lambda(z))\hspace{0.2cm}and\hspace{0.2cm} \tilde{\Delta}F=0\}
=\{F\in D(\tilde{\Delta}),\hspace{0.2cm}\tilde{\Delta}F=0\}.
\end{align}
Conversely, let $F\in D(\tilde{\Delta})$ such that  $\tilde{\Delta}F=0=H_{1}F$. Then, we have
\begin{align}\label{E3.27}
\frac{\partial F}{\partial z}\in L^{2}(\mathbb{D},d\lambda(z))\hspace{0.2cm}and\hspace{0.2cm} F\in\{F\in L^{2}(\mathbb{D},d\lambda(z)),\hspace{0.2cm}H_{1}F=0\}=\mathcal{A}^{2,1}_{0}(\mathbb{D}).
\end{align}
Finally, by using $(\ref{E3.25})$ the holomorphy of $F$ follows. Then, we get $F$ is an element of the Dirichlet space. Hence, we obtain
\begin{align}\label{E3.28}
 \mathcal{D}=\{F\in D(\tilde{\Delta}),\hspace{0.2cm}\tilde{\Delta}F=0\}.
 \end{align}
 The proof is closed.
\end{proof}
Moreover, we can give some spectral results for the operator $\tilde{\Delta}$. Precisely, by using the same notations as in the proposition $(\ref{P3.1})$, we can state the following.
\begin{proposition}\label{P3.2}
\begin{description}
  \item[(i)] The operator $\tilde{\Delta}$ is closable and admits a self-adjoint extension.
  \item[(ii)] The operator $\tilde{\Delta}$ is an unbounded non self-adjoint operator.
  \item[(iii)] $0$ belongs to the point spectrum of  $\tilde{\Delta}.$
\end{description}
\end{proposition}
\begin{proof}
We return back to the self-adjoint of operator $H_{\nu}$ defined by $(\ref{E2.57})$ and  $(\ref{E2.58})$, in which we set $\nu=1$. It is easy to see that the operator $H_{1}$ is an extension of the operator $\tilde{\Delta}$ (See \cite[p.4]{Kon} for general theory). We write this fact by
\begin{align}\label{E3.29}
\tilde{\Delta}\subset H_{1}.
\end{align}
Thus $(i)$ is proved.\\
It is not hard to see that $D(\tilde{\Delta})\neq D(H_{1}).$ Then, if we suppose that  $\tilde{\Delta}$ is self-adjoint, we obtain from $(\ref{E3.29})$ and the proposition $(1.6)$ in \cite[p.9]{Kon} that
\begin{align}\label{E3.30}
H_{1}=(H_{1})^{*}\subset(\tilde{\Delta})^{*}=\tilde{\Delta}.
\end{align}
Then, by $(\ref{E3.29})$ and $(\ref{E3.30})$, we get $H_{1}=\tilde{\Delta}$ which implies that $D(\tilde{\Delta})=D(H_{1}).$ This contradicts the fact $D(\tilde{\Delta})\neq D(H_{1})$. Thus, $(ii)$ is proved.\\
The point $(iii)$ is just an other way to state the proposition $(\ref{P3.1})$. The proof is closed.
\end{proof}
\subsection{Bargmann transform associated with Dirichlet space}
Now, we turn to the Dirichlet space. The usual convenient norm associated with this space is given by
\begin{align}\label{E3.31}
 \mathcal{N}(f)=\mid f(0) \mid^{2}+\int_{\mathbb{D}}\mid f^{'}(z)\mid^{2}d\lambda(z).
\end{align}
The corresponding reproducing kernel \cite[p.51]{Arc} reads as
\begin{align}\label{E3.32}
K(z,w)=\frac{1}{\pi}(1+\log(\frac{1}{1-z\overline{w}})).
\end{align}
The following result gives a Bargmann transform associated with the classical Dirichlet space.
\begin{proposition}\label{P3.3} For the classical Dirichlet space, we have the following associated Bargmann transform
 \begin{align}\label{E3.33}
 \nonumber L^{2}(\mathbb{R}_{+},& e^{-x}dx)\longrightarrow \mathcal{D}\\
 &f\longmapsto B[f](z):=\int_{0}^{+\infty}K(z,x)f(x)dx,
 \end{align}
  where,
 \begin{align}\label{E3.34}
 K(z,x)=\frac{1}{\sqrt{\pi}}[1+\frac{z}{\Gamma(\frac{3}{2})}
 \int_{0}^{+\infty}\sqrt{t}e^{-t}(1-ze^{-t})^{-2}\exp(\frac{-xze^{-t}}{1-ze^{-t}})L_{1}(\frac{x}{1-ze^{-t}})dt].
 \end{align}
  \end{proposition}
  \begin{proof}  In order to build the above isometry, we consider
  $M_{1}=\mathbb{R}_{+},$ $d\mu_{1}(x)=e^{-x}dx,$ and $\mathcal{A}=\mathcal{D},$ where $<,>_{\mathcal{A}}$ is the scalar product associated with the norm defined in $(\ref{E3.31})$.\\
  We return to the equation $(\ref{E2.47})$ and we set $\delta=0$. Then, an orthonormal basis of the Hilbert space $L^{2}(\mathbb{R}_{+},e^{-x}dx)$ can be given by the Laguerre polynomials
  \begin{align}\label{E3.35}
  \varphi_{j}(x)=L_{j}(x),\hspace{0.25cm}j\in\mathbb{Z}.
  \end{align}
  An orthonormal basis of the Dirichlet space $\mathcal{D}$ is
  \begin{align}\label{E3.36}
   \psi_{j}(z)=\left\{
                 \begin{array}{ll}
                   \sqrt{\frac{1}{\pi}};&j=0\hbox{} \\\\
                   \sqrt{\frac{1}{\pi j}}z^{j};&j\geq 1. \hbox{}
                 \end{array}
               \right.
\end{align}
The proof is easy and can be omitted.\\
From $(ii)$ of theorem $(\ref{T2.2})$, the following series
\begin{align}\label{E3.37}
K(z,x)=\sum_{j=0}^{+\infty}\overline{\varphi_{j}}(x)\psi_{j}(z),
\end{align}
 converges $a.e$ $d\mu_{1}(x)$ for each $z\in \mathbb{D}.$ \\
 By replacing $\varphi_{j}(x)$ and $\psi_{j}(z)$ by its above expressions given in $(\ref{E3.35})$ and $(\ref{E3.36})$, the kernel $K(z,x)$ can be written as
\begin{align}\label{E3.38}
   \nonumber K(z,x)&=\frac{1}{\sqrt{\pi}}[1+\sum_{j=1}^{+\infty}\frac{1}{\sqrt{j}}z^{j}L_{j}(x)]\\
   \nonumber&=\frac{1}{\sqrt{\pi}}[1+z\sum_{j=0}^{+\infty}\frac{1}{\sqrt{j+1}}z^{j}L_{j+1}(x)]\\
  \nonumber&=\frac{1}{\sqrt{\pi}}[1+z\sum_{j=0}^{+\infty}\frac{1}{\sqrt{j+1}}[\frac{1}{j+1}][\frac{(j+1)!}{j!}z^{j}L_{j+1}(x)]\\
  &=\frac{1}{\sqrt{\pi}}[1+z\sum_{j=0}^{+\infty}\frac{1}{(j+1)\frac{3}{2}}[\frac{(j+1)!}{j!}z^{j}L_{j+1}(x)].
\end{align}
 Using the formula given in \cite[p.42]{Sch}
 \begin{align}\label{E3.39}
 \frac{1}{\lambda^{s}}=\frac{1}{\Gamma(s)}\mathscr{L}(t^{s-1})(\lambda):=\frac{1}{\Gamma(s)}\int_{0}^{+\infty}e^{-\lambda t}t^{s-1}dt,\hspace{0.25cm} Re(\lambda)>0,
 \end{align}
  for $\lambda=j+1$ and $s=\frac{3}{2}$, the formula $(\ref{E3.38})$ can be rewritten as
\begin{align}\label{E3.40}
  \nonumber K(z,x)&=\frac{1}{\sqrt{\pi}}[1+\frac{z}{\Gamma(\frac{3}{2})}\sum_{j=0}^{+\infty}
   \int_{0}^{+\infty}\sqrt{t}e^{-(j+1)t}dt
 [\frac{(j+1)!}{j!}z^{j}L_{j+1}(x)]\\
 &=\frac{1}{\sqrt{\pi}}[1+\frac{z}{\Gamma(\frac{3}{2})}\sum_{j=0}^{+\infty}\int_{0}^{+\infty}\sqrt{t}e^{-t}(ze^{-t})^{j}
 dt[\frac{(j+1)!}{j!}z^{j}L_{j+1}(x)].
   \end{align}
To compute the above sum, we consider the following series
\begin{align}\label{E3.41}
S_{a}(z,x)=\sum_{j=0}^{+\infty}\int_{a}^{+\infty}\sqrt{t}e^{-t}(ze^{-t})^{j}
 [\frac{(j+1)!}{j!}z^{j}L_{j+1}]dt,\hspace{0.25cm}a\geq0.
\end{align}
Note that
\begin{align}\label{E3.42}
K(z,x)=\frac{1}{\sqrt{\pi}}[1+\frac{z}{\Gamma(\frac{3}{2})}S_{0}(z,x)].
\end{align}
First, we have to prove that the series $S_{a}(z,x)$ converge, for all $a>0$. To do so, we recall the following asymptotic formula for the Laguerre polynomials \cite[p.248]{Mag}
 \begin{align}\label{E3.43}
 L_{j}^{(\beta)}=x^{\frac{-\beta}{2}-\frac{1}{4}}O(j^{\frac{\beta}{2}-\frac{1}{4}}),
 \end{align}
  as $j\longrightarrow+\infty$, for $\frac{c}{j}\leq x\leq w$ where $c$ and $w$ are fixed positive constants. Then, for $\beta=0$ and $j$ enough large $j\geq j_{0}$, we obtain the following estimate
  \begin{align}\label{E3.44}
 \mid L_{j}(x)\mid\leq M \mid j x\mid^{\frac{-1}{4}},\hspace{0.25cm}j\geq j_{0},
 \end{align}
 where $M$ is a positive constant. Then, by $(\ref{E3.44})$ and for $k\in\mathbb{Z}_{+}$ enough large with $t>a$, we get the following inequality
\begin{align}\label{E3.45}
    \nonumber\mid\sum_{j=0}^{k}(ze^{-t})^{j}(j+1)L_{j+1}(x)\mid&\leq \sum_{j=0}^{j_{0}-1}(j+1)L_{j+1}(x)+\sum_{j=j_{0}}^{k}e^{-aj}M
   \mid x\mid^{\frac{-1}{4}}(j+1)^{\frac{3}{4}}\\
    \nonumber&\leq\sum_{j=0}^{j_{0}-1}(j+1)L_{j+1}(x)+M
   \mid x\mid^{\frac{-1}{4}}\sum_{j=j_{0}}^{+\infty}e^{-aj}(j+1)^{\frac{3}{4}}\\
   &=C_{x,a}<+\infty.
   \end{align}
Applying the Lebesgue dominated convergence theorem, we can interchange the sum and the integral in $(\ref{E3.41})$ in which we consider $a>0$. Then, we obtain
\begin{align}\label{E3.46}
S_{a}(z,x)=\int_{a}^{+\infty}\sqrt{t}e^{-t}\sum_{j=0}^{+\infty}(ze^{-t})^{j}
 [\frac{(j+1)!}{j!}z^{j}L_{j+1}]dt,\hspace{0.25cm}a>0.
  \end{align}
 Now, using the following generating function \cite[p.104]{Mou}
 \begin{align}\label{E3.47}
 \sum_{j=0}^{+\infty}\frac{(j+k)!}{k!j!}L^{(\beta)}_{j+k}(y)s^{j}=
(1-s)^{-\beta-k-1}\exp(\frac{-ys}{1-s})L^{(\beta)}_{k}(\frac{y}{1-s}),
\end{align}
for  $k=1$, $\beta=0$, $y=x$ and $s=ze^{-t}$, the equation $(\ref{E3.46})$ becomes
 \begin{align}\label{E3.48}
 S_{a}(z,x)=\int_{a}^{+\infty}\sqrt{t}e^{-t}(1-ze^{-t})^{-2}\exp(\frac{-xze^{-t}}{1-ze^{-t}})L_{1}(\frac{x}{1-ze^{-t}})dt.
 \end{align}
 We rewrite the above function $S_{a}(z,x)$ as
 \begin{align}\label{E3.49}
  S_{a}(z,x)=\int_{0}^{+\infty}1_{[a,+\infty[}\sqrt{t}e^{-t}(1-ze^{-t})^{-2}\exp(\frac{-xze^{-t}}{1-ze^{-t}})L_{1}(\frac{x}{1-ze^{-t}})dt,
  \end{align}
 where $1_{[a,+\infty[}$ is the characteristic function associated with the set $[a,+\infty[\subset\mathbb{R}_{+}$. Then,
 by the continuity of the functions $s\longmapsto e^{s}$ and $s\longmapsto L_{1}(s)$, the involved function in the right hand side of $(\ref{E3.49})$ satisfies the following estimate
 \begin{align}\label{E3.50}
  \mid 1_{[a,+\infty[}\sqrt{t}e^{-t}(1-ze^{-t})^{-2}\exp(\frac{-xze^{-t}}{1-ze^{-t}})L_{1}(\frac{x}{1-ze^{-t}})\mid\leq \tilde{C}_{z,x}\sqrt{t}e^{-t}, \hspace{0.25cm}t\geq0,
 \end{align}
  where $\tilde{C}_{z,x}$ is a positive constant.\\
By making appeal of Lebesgue dominated convergence theorem, we get
\begin{align}\label{E3.51}
  \lim_{a\longrightarrow 0}S_{a}(z,x)=\int_{0}^{+\infty}\sqrt{t}e^{-t}(1-ze^{-t})^{-2}\exp(\frac{-xze^{-t}}{1-ze^{-t}})L_{1}(\frac{x}{1-ze^{-t}})dt.
  \end{align}
Hence, we obtain
 \begin{align}\label{E3.52}
 S_{0}(z,x)=\int_{0}^{+\infty}\sqrt{t}e^{-t}(1-ze^{-t})^{-2}\exp(\frac{-xze^{-t}}{1-ze^{-t}})L_{1}(\frac{x}{1-ze^{-t}})dt.
  \end{align}
  Finally, returning back to the equation $(\ref{E3.42})$, we get the desired kernel
\begin{align}\label{E3.53}
 K(z,x)=\frac{1}{\sqrt{\pi}}+\frac{z}{\sqrt{\pi}\Gamma(\frac{3}{2})}\int_{0}^{+\infty}\sqrt{t}e^{-t}
 (1-ze^{-t})^{-2}\exp(\frac{-xze^{-t}}{1-ze^{-t}})L_{1}(\frac{x}{1-ze^{-t}})dt.
   \end{align}
\end{proof}
\begin{remark}
By considering the following natural isometry
\begin{align}\label{E3.54}
		\nonumber T:\hspace{0.25cm}L^{2}(\mathbb{R}_{+},dx)&\longrightarrow L^{2}(\mathbb{R}_{+},e^{-x}dx)\\
		&f\longmapsto T(f)(x)=e^{\frac{x}{2}}f(x),
\end{align}
we obtain in a canonical way from the above proposition the following isometry
\begin{align}\label{E3.55}
 \nonumber L^{2}(\mathbb{R}_{+},& dx)\longrightarrow \mathcal{D}\\
 &f\longmapsto B[f](z):=\int_{0}^{+\infty}K(z,x)f(x)dx,
 \end{align}
  where,
  \begin{align}\label{E3.56}
 K(z,x)=\frac{e^{\frac{x}{2}}}{\sqrt{\pi}}[1+\frac{z}{\Gamma(\frac{3}{2})}\int_{0}^{+\infty}\sqrt{t}e^{-t}
 (1-ze^{-t})^{-2}\exp(\frac{-xze^{-t}}{1-ze^{-t}})L_{1}(\frac{x}{1-ze^{-t}})dt].
   \end{align}
\end{remark}
\section{Generalized Bergman-Dirichlet space and associated Bargmann transform}
\subsection{Generalized Bergman-Dirichlet space as harmonic space}
In this section, we intend to associate a new Bargmann transforms to a class of generalized Bergman-Dirichlet space $\mathcal{D}_{m}^{\alpha}$, $\alpha>-1$, $m\in\mathbb{Z}_{+}$ called weighted Bergman-Dirichlet space of order $m$. The functional space $\mathcal{D}_{m}^{\alpha}$ have been considered by Ahmed Intissar \emph{et} al. in \cite{Elh}. In order to avoid any confusion, it should be noted that the space $\mathcal{D}_{m}^{\alpha}$ was denoted in the last reference by $\mathcal{A}^{2,\alpha}_{m}(\mathbb{D}_{R})$, with $R=1$.\\
To introduce the  weighted Bergman-Dirichlet space  $\mathcal{D}_{m}^{\alpha}$ of order $m$, we need to fixe some notations.\\
Let $\alpha>-1$ and let $L^{2}(\mathbb{D},d\mu_{\alpha})=L^{2,\alpha}(\mathbb{D})$ the space of complex valued functions on $\mathbb{D}$ that are square integrable with respect to the density measure\\
 $d\mu_{\alpha}(z)=(1-\mid z\mid^{2})^{\alpha}d\lambda(z).$ The space $L^{2,\alpha}(\mathbb{D})$ is a Hilbert space with the norm
\begin{align}\label{E4.1}
\parallel f\parallel_{\alpha}^{2}=\int_{\mathbb{D}}\mid f(z)\mid^{2}d\mu_{\alpha}(z),
\end{align}
corresponding to the hermitian scalar product
\begin{align}\label{E4.2}
<f,g>_{\alpha}=\int_{\mathbb{D}}f(z)\overline{g(z)}d\mu_{\alpha}(z).
\end{align}
$Hol(\mathbb{\mathbb{D}})$ denotes the vector space of all convergent entire series $\displaystyle{f(z)=\sum_{j=0}^{+\infty}}a_{j}z^{j}$ on $\mathbb{D}$. For any arbitrary non negative $m\in\mathbb{Z}_{+}$, a function $f\in Hol(\mathbb{\mathbb{D}})$ splits as
\begin{align}\label{E4.3}
f(z)=f_{1,m}(z)+f_{2,m}(z),
\end{align}
where
$\displaystyle{f_{1,m}(z)=\sum_{j=0}^{m-1}a_{j}z^{j}}$ and $\displaystyle{f_{2,m}(z)=\sum_{j=m}^{+\infty}a_{j}z^{j}},$
with the convention that $f_{1,m}(z)=0$  when $m=0.$\\
Note that we have
\begin{align}\label{E4.4}
f^{(m)}(z)=f_{2,m}^{(m)}(z).
\end{align}
Therefore, the space $\mathcal{D}_{m}^{\alpha}$ is defined as the space of holomorphic functions
\begin{align}\label{E4.5}
\mathcal{D}_{m}^{\alpha}=\{f\hspace{0.2cm}\mathbb{D}\longrightarrow \mathbb{C}\hspace{0.2cm} holomorphic ,\hspace{0.2cm}\int_{\mathbb{D}}\mid f^{(m)}(z)\mid^{2}d\mu_{\alpha}(z)<+\infty\},
\end{align}
endowed with the following scalar product
\begin{align}\label{E4.6}
<f,g>_{\alpha,m}=<f_{1,m},g_{1,m}>_{\alpha}+<f_{2,m},g_{2,m}>_{\alpha},
\end{align}
for given $f,g\in \mathcal{D}_{m}^{\alpha}$.\\
According to \cite{Elh}, the space  $\mathcal{D}_{m}^{\alpha}$ with $\alpha>-1$ is non trivial and an element\\ $\displaystyle{f(z)=\sum_{j=0}^{+\infty}a_{j}z^{j}}$ belongs to $\mathcal{D}_{m}^{\alpha}$ if and only if
 \begin{align}\label{E4.7}
 \displaystyle{\sum_{j=m}^{+\infty}}\frac{(j!)^{2}}{(j-m)!\Gamma(j-m+\alpha+2)}\mid a_{j}\mid^{2}<+\infty.
 \end{align}
 Moreover, the space $\mathcal{D}_{m}^{\alpha}$ is a R.K.H.S and its reproducing kernel is given by \cite{Elh}
 \begin{align}\label{E4.8}
 K_{m}^{\alpha}(z,w)=\frac{\alpha+1}{\pi}\{\sum_{j=0}^{m-1}\frac{(\alpha+2)_{j}(z\overline{w})^{j}}{j!}+
 \frac{(z\overline{w})^{m}}{(m!)^{2}}\prescript{}{3}{F}_2^{}(^{1,1,\alpha+2}_{j+1,j+1}\mid z\overline{w})\}.
 \end{align}
For $\alpha=0$ and $m=1$, the corresponding reproducing kernel reduces further to be reproducing kernel of the classical Dirichlet space $\mathcal{D}$
\begin{align}\label{E4.9}
 K_{1}^{0}(z,w)=\frac{1}{\pi}(1+z\overline{w}\prescript{}{2}{F}_1^{}(^{1,1}_{2}\mid z\overline{w})\}=\frac{1}{\pi}(1+\log(\frac{1}{1-z})).
 \end{align}
 Before building an associated Bargmann transform with the generalized Dirichlet space $\mathcal{D}_{m}^{\alpha}$, we will give, in the same way as in  the case of the  classical Dirichlet space, a characterization of the functional space $\mathcal{D}_{m}^{\alpha}$ as harmonic space of a single elliptic partial differential operator. Precisely, we have the following proposition
 \begin{proposition}\label{P4.1}
Let $\alpha>-1$, $m\in\mathbb{Z}_{+}$ and $\tilde{\Delta}_{\alpha}$ be the partial differential operator defined by
\begin{align}\label{E4.10}
 \tilde{\Delta}_{\alpha}:=-4(1-\mid z\mid^{2})[(1-\mid z\mid^{2})\frac{\partial^{2}}{\partial z\partial\overline{z}}-(\alpha+2)\overline{z}\frac{\partial}{\partial\overline{z}}].
 \end{align}
 It acts on the Hilbert space $L^{2,\alpha}(\mathbb{D})=L^{2}(\mathbb{D},(1-\mid z\mid^{2})^{\alpha}d\lambda(z))$,
 with the domain
 \begin{align}\label{E4.11}
 D(\tilde{\Delta}_{\alpha}):=\{F\in L^{2,\alpha}(\mathbb{D}),\hspace{0.25cm}\tilde{\Delta}_{\alpha}F\in L^{2,\alpha}(\mathbb{D})\hspace{0.25cm}and\hspace{0.25cm}\frac{\partial^{m} F}{\partial z^{m}}\in L^{2,\alpha}(\mathbb{D})\}.
 \end{align}
 Then, we have
 \begin{align}\label{E4.12}
 \mathcal{D}_{m}^{\alpha}=\{F\in D(\tilde{\Delta}_{\alpha}),\hspace{0.25cm}\tilde{\Delta}_{\alpha}F=0\}.
 \end{align}
 \end{proposition}
\begin{proof}Let $\displaystyle{f(z)=\sum_{j=0}^{+\infty}a_{j}z^{j}}$ be a holomorphic function on $\mathbb{D}$.\\
 By using the same method as for the computation of the integral appearing in $(\ref{E3.17})$, we can give by a direct computation the following  formula for the square norm of $f$ in $L^{2,\alpha}(\mathbb{D})$
\begin{align}\label{E4.13}
\int_{\mathbb{D}}\mid f(z)\mid^{2}(1-\mid z\mid^{2})^{\alpha}d\lambda(z)
=\pi\sum_{j=0}^{+\infty}\frac{\Gamma(j+1)\Gamma(\alpha+1)}{\Gamma(j+\alpha+2)}\mid a_{j}\mid^{2}.
\end{align}
Then, the following weighted Bergman space on $\mathbb{D}$
\begin{align}\label{E4.14}
\mathcal{A}^{2,\alpha}(\mathbb{D})=\{f\hspace{0.2cm}\mathbb{D}\longrightarrow \mathbb{C}\hspace{0.2cm} holomorphic ,\hspace{0.2cm}\int_{\mathbb{D}}\mid f(z)\mid^{2}(1-\mid z\mid^{2})^{\alpha}d\lambda(z)<+\infty\}
\end{align}
can be rewritten as
\begin{align}\label{E4.15}
\mathcal{A}^{2,\alpha}(\mathbb{D})=\{\displaystyle{f(z)=\sum_{j=0}^{+\infty}a_{j}z^{j}}\hspace{0.2cm} holomorphic \hspace{0.2cm}on\hspace{0.2cm} \mathbb{D}\hspace{0.2cm}and\hspace{0.2cm}\pi\sum_{j=0}^{+\infty}\frac{\Gamma(j+1)\Gamma(\alpha+1)}{\Gamma(j+\alpha+2)}\mid a_{j}\mid^{2}<+\infty\}.
\end{align}
It is well known that the weighted Bergman space $\mathcal{A}^{2,\alpha}(\mathbb{D})$ defined in $(\ref{E4.14})$ is a R.K.H.S, with
the reproducing kernel \cite{Hed,Zhu2}
\begin{align}\label{E4.16}
K(z,w)=\frac{\alpha+1}{\pi}(1-z\overline{w})^{-\alpha-2}.
\end{align}
Now, by using the asymptotic formula \cite[p.22]{Bea}
\begin{align}\label{E4.17}
\frac{\Gamma(z+a)}{\Gamma(z)}\sim z^{a}\hspace{0.2cm} when \hspace{0.2cm}\mid z\mid\longrightarrow+\infty,\hspace{0.25cm}a\in\mathbb{C},
\end{align}
for $z=j+1$ and $a=\alpha+1$, we obtain
\begin{align}\label{E4.18}
\pi\frac{\Gamma(j+1)\Gamma(\alpha+1)}{\Gamma(j+\alpha+2)}\sim \pi \Gamma(\alpha+1) (1+j)^{-\alpha-1}
\hspace{0.2cm} as \hspace{0.2cm}j\longrightarrow+\infty .
\end{align}

Thus, a function $\displaystyle{f(z)=\sum_{j=0}^{+\infty}a_{j}z^{j}}$ belongs to $\mathcal{A}^{2,\alpha}(\mathbb{D})$ if and only if
\begin{align}\label{E4.19}
\sum_{j=0}^{+\infty}(1+j)^{-\alpha-1}\mid a_{j}\mid^{2}<+\infty.
\end{align}
Also, by using the asymptotic formula $(\ref{E4.17})$, the membership test, given in  $(\ref{E4.7})$ for the functional space $\mathcal{D}_{m}^{\alpha}$ can be rewritten as
 \begin{align}\label{E4.20}
 \displaystyle{f(z)=\sum_{j=0}^{+\infty}a_{j}z^{j}}\in \mathcal{D}_{m}^{\alpha}\Longleftrightarrow \sum_{j=m}^{+\infty}(1+j)^{2m-\alpha-1}\mid a_{j}\mid^{2}<+\infty.
 \end{align}
Considering  $(\ref{E4.19})$ and $(\ref{E4.20})$ with the following inequality
\begin{align}\label{E4.21}
\sum_{j=m}^{+\infty}(1+j)^{-\alpha-1}\mid a_{j}\mid^{2}\leq\sum_{j=m}^{+\infty}(1+j)^{2m-\alpha-1}\mid a_{j}\mid^{2},
\end{align}
we obtain the following inclusion
\begin{align}\label{E4.22}
\mathcal{D}_{m}^{\alpha}\subset\mathcal{A}^{2,\alpha}(\mathbb{D}).
\end{align}
Now, we return back to the partial differential operator defined by $(\ref{E2.57})$ and  $(\ref{E2.58})$ in which we choose the parameter $\nu=\frac{\alpha}{2}+1$. Then, we get the following linear partial differential operator
\begin{align}\label{E4.23}
 \Delta_{\alpha}:=H_{\frac{\alpha}{2}+1}=-4(1-\mid z\mid^{2})[(1-\mid z\mid^{2})\frac{\partial^{2}}{\partial z\partial\overline{z}}-(\alpha+2)\overline{z}\frac{\partial}{\partial\overline{z}}]
\end{align}
acting on the Hilbert space $L^{2,\alpha}(\mathbb{D})=L^{2}(\mathbb{D},(1-\mid z\mid^{2})^{\alpha}d\lambda(z))$,
 with the maximal domain
 \begin{align}\label{E4.24}
 D(\Delta_{\alpha}):=\{F\in L^{2,\alpha}(\mathbb{D}),\hspace{0.25cm}\Delta_{\alpha}F\in L^{2,\alpha}(\mathbb{D})\}.
 \end{align}
From $(\ref{E2.59})$, it follows that the point spectrum of $\Delta_{\alpha}$ is given by
\begin{align}\label{E4.25}
\sigma(\Delta_{\alpha})=\{4l(\alpha-l+1),\hspace{0.25cm}l=0,1,2,...,[\frac{\alpha-1}{2}]\}.
\end{align}
The corresponding eigenespaces
\begin{align}\label{E4.26}
\mathcal{E}_{\ell}^{\alpha}(\mathbb{D}):=\mathcal{A}_{\ell}^{2,\frac{\alpha}{2}+1}(\mathbb{D})
\end{align}
associated with the eigenvalue $E_{\alpha}(\ell)=4\ell(\alpha-\ell+1)$ is a R.K.H.S with the following reproducing kernel
\begin{align}\label{E4.27}
 \nonumber\tilde{K}_{\ell}^{\alpha}(z,w)&:=K_{\ell}^{\frac{\alpha}{2}+1}(z,w)\\
 \nonumber&=(\frac{\alpha+1-2\ell}{\pi})(1-z\overline{w})^{-\alpha-2}(\frac{\mid1-z\overline{w}\mid^{2}}{(1-\mid z\mid^{2})(1-\mid w\mid^{2})})^{\ell}\\
&\times P_{\ell}^{(0,\alpha+1-2\ell)}(2\frac{(1-\mid z\mid^{2})(1-\mid w\mid^{2})}{\mid1-z\overline{w}\mid^{2}}-1),
\end{align}
where we have used $(\ref{E2.60})$.\\
It follows that, for $\ell=0$, the reproducing kernel of the eigenspace $\mathcal{E}_{0}^{\alpha}(\mathbb{D})$  is reduced to
\begin{align}\label{E4.28}
\tilde{K}_{0}^{\alpha}(z,w)=\frac{\alpha+1}{\pi}(1-z\overline{w})^{-\alpha-2}
\end{align}
being the reproducing kernel of the Bergman space $\mathcal{A}^{2,\alpha}(\mathbb{D})$. From $(\ref{E4.16})$ and $(\ref{E4.28})$, we see that the both spaces $\mathcal{A}^{2,\alpha}(\mathbb{D})$ and  $\mathcal{E}_{0}^{\alpha}(\mathbb{D})$
have the same reproducing kernel. Thus, by the proposition $(3.3)$ in \cite[p.9]{Pau}, we obtain the following equality
\begin{align}\label{E4.29}
\mathcal{E}_{0}^{\alpha}(\mathbb{D}):=\{F\in L^{2,\alpha}(\mathbb{D}),\hspace{0.25cm}\Delta_{\alpha}F=0\}=
\mathcal{A}^{2,\alpha}(\mathbb{D}).
\end{align}
Then, the inclusion  $(\ref{E4.22})$ becomes
\begin{align}\label{E4.30}
\mathcal{D}_{m}^{\alpha}\subset\{F\in L^{2,\alpha}(\mathbb{D}),\hspace{0.25cm}\Delta_{\alpha}F=0\}.
\end{align}
By using the holomorphy of an element $F$ in $\mathcal{D}_{m}^{\alpha}$, we have $\frac{\partial F}{\partial \overline{z}}=0$, then we obtain\\
$\Delta_{\alpha}F=0$. Consequently, we get the following inclusion
 \begin{align}\label{E4.31}
\mathcal{D}_{m}^{\alpha}&\subset\{F\in L^{2,\alpha}(\mathbb{D}),\hspace{0.25cm}\frac{\partial^{m} F}{\partial z^{m}}\in L^{2,\alpha}(\mathbb{D}),\hspace{0.25cm}\Delta_{\alpha}F=0\}
=\{F\in D(\tilde{\Delta}_{\alpha}),\hspace{0.25cm}\tilde{\Delta}_{\alpha}F=0\}.
\end{align}
Conversely, let $F\in D(\tilde{\Delta}_{\alpha})$ such that $\tilde{\Delta}_{\alpha}F=0=\Delta_{\alpha}F.$ Then, with the help of $(\ref{E4.29})$, we have
\begin{align}\label{E4.32}
F\in\{F\in L^{2,\alpha}(\mathbb{D}),\hspace{0.25cm}\Delta_{\alpha}F=0\}=\mathcal{A}^{2,\alpha}(\mathbb{D}).
\end{align}
Thus, the holomorphy of $F$ follows.\\
Then, $F$ is holomorphic on $\mathbb{D}$, $F\in L^{2,\alpha}(\mathbb{D})$ and $\frac{\partial^{m} F}{\partial z^{m}}\in L^{2,\alpha}(\mathbb{D}).$ This means that $F$ is an element of  the generalized Dirichlet space $\mathcal{D}_{m}^{\alpha}.$ Hence, we get
\begin{align}\label{E4.33}
\{F\in D(\tilde{\Delta}_{\alpha}),\hspace{0.25cm}\tilde{\Delta}_{\alpha}F=0\} \subset\mathcal{D}_{m}^{\alpha}.
\end{align}
The inclusions $(\ref{E4.31})$ and $(\ref{E4.33})$  end the proof.
\end{proof}
In the same way as in the  proposition $(\ref{P3.2})$, we can state the following result.
\begin{proposition}\label{P4.2}
\begin{description}
  \item[(i)] The operator $\tilde{\Delta}_{\alpha}$ is closable and admits a self-adjoint extension.
  \item[(ii)] The operator $\tilde{\Delta}_{\alpha}$ is an unbounded non self-adjoint operator.
  \item[(iii)] $0$ belongs to the point spectrum of  $\tilde{\Delta}_{\alpha}.$
\end{description}
\end{proposition}
The proof is the same as for the proposition $(\ref{P3.2})$ and can be omitted.
\subsection{Bargmann transform associated with generalized Bergman-Dirichlet space}
For the generalized Dirichlet space $\mathcal{D}_{m}^{\alpha}$, we shall associate a Bargmann transform. Precisely, we have the following proposition.
\begin{proposition}\label{P4.2} Let $\alpha>-1$ and $m\in\mathbf{Z}_{+}$, $m\geq2$. Then, we have the following unitary isomorphism
\begin{align}\label{E4.34}
 \nonumber L^{2}(\mathbb{R}_{+},& x^{\alpha}e^{-x}dx)\longrightarrow \mathcal{D}_{m}^{\alpha}\\
 &f\longmapsto B[f](z):=\int_{0}^{+\infty}K(z,x)f(x)dx,
 \end{align}
  where
 \begin{align}\label{E4.35}
  \nonumber  K(z,x)&=\frac{1}{\sqrt{\pi \Gamma(1+\alpha)}}\displaystyle{\sum_{0\leq j<m}}z^{j}L_{j}^{(\alpha)}(x)\\
   &\hspace{-1.5cm}+\frac{m!z^{m}(\Gamma(\frac{3}{2}))^{-m}\Gamma(\frac{1}{2}))^{1-m}}{\sqrt{\pi \Gamma(1+\alpha)}}
   \int_{0}^{+\infty}\omega_{(\alpha,m)}(t)(1-ze^{-t})^{-\alpha-m-1}
 \exp(\frac{-xze^{-t}}{1-ze^{-t}})L_{m}^{(\alpha)}(\frac{x}{1-ze^{-t}})dt,
   \end{align}
with $\varpi_{(\alpha,m)}$ is the function defined as follows
\begin{align}\label{E4.36}
\varpi_{(\alpha,m)}(t)=\sqrt{t}e^{-t}*[(\sqrt{t}e^{-2t})*(\frac{e^{-(\alpha+2)t}}{\sqrt{t}})]*...
*[(\sqrt{t}e^{-mt})*(\frac{e^{-(\alpha+m)t}}{\sqrt{t}})], \hspace{0.25cm}m\geq2.
\end{align}
The notation $f*g$ means the following convolution product \cite[p.91]{Sch}
\begin{align}\label{E4.37}
f*g(x)=\int_{0}^{x}f(x-y)g(y)dy.
\end{align}
\end{proposition}
In order to prove the above proposition, we need the following key lemma.\\
With the help of the notations given in the proposition  $(\ref{P4.2})$, we have the following result.
\begin{lemma} \label{L4.1}Let $m\in\mathbb{Z}^{+}$, $m\geq2$ and $\alpha>-1$, then
\begin{enumerate}
  \item The following estimate holds
  \begin{align}\label{E4.38}
  \varpi_{(\alpha,m)}(t)\leq ({\cal{B}}(\frac{3}{2},\frac{1}{2}))^{m-1}t^{m}\sqrt{t}e^{-t},
  \end{align}
  where ${\cal{B}}(x,y)$ denotes the beta special function.
  \item The Laplace transform of $\varpi_{(\alpha,m)}(t)$ is well defined. Moreover, we have
\begin{align}\label{E4.39}
\mathscr{L}(\varpi_{(\alpha,m)}(t))(j)=\frac{(\Gamma(\frac{3}{2}))^{m}}{[(j+1)(j+2)...(j+m)]\frac{3}{2}}
\frac{(\Gamma(\frac{1}{2}))^{m-1}}{[(j+\alpha+2)(j+\alpha+3)...(j+\alpha+m)]\frac{1}{2}},
\end{align}
where $\mathscr{L}$ denotes the classical Laplace transform defined by \cite[p.2]{Sch}
\begin{align}\label{E4.40}
\mathscr{L}(f(t))(\lambda):=\int_{0}^{+\infty}e^{-\lambda t}f(t)dt, \hspace{0.25cm}\lambda>0.
\end{align}
\end{enumerate}
\end{lemma}
\begin{proof}
 For all $\ell\in\{2,3,...,m\}$, we introduce the following function
  \begin{align}\label{E4.41}
\phi_{(\alpha,\ell)}(t)=(\sqrt{t}e^{-\ell t})*(\frac{e^{-(\alpha+\ell)t}}{\sqrt{t}}).
\end{align}
Explicitly, we have
\begin{align}\label{E4.42}
\nonumber\phi_{(\alpha,\ell)}(t)&=(\sqrt{t}e^{-\ell t})*(\frac{e^{-(\alpha+\ell)t})}{\sqrt{t}})(t)\\
\nonumber&=\int_{0}^{t}\frac{\sqrt{t-s}e^{-\ell t+s-(\alpha+1)s}}{\sqrt{s}}ds\\
\nonumber&\leq e^{(1-\ell)t}\int_{0}^{t}\frac{\sqrt{t-s}}{\sqrt{s}}ds\\
\nonumber&= e^{(1-\ell)t}\int_{0}^{1}\frac{\sqrt{1-x}}{\sqrt{x}}dx,\hspace{0.25cm}x=\frac{s}{t}\\
&= {\cal{B}}(\frac{3}{2},\frac{1}{2})te^{(1-\ell)t}.
\end{align}
For $\ell\in\{2,3,...,m\}$, we can rewrite the function  $\varpi_{(\alpha,\ell)}$, defined  in $(\ref{E4.36})$, as follows
\begin{align}\label{E4.43}
\varpi_{(\alpha,\ell)}(t)=\sqrt{t}e^{-t}*\phi_{(\alpha,2)}(t)*...*\phi_{(\alpha,\ell)}(t).
\end{align}
For $\ell=2$, we obtain
\begin{align}\label{E4.44}
\nonumber\varpi_{(\alpha,2)}(t)&=\sqrt{t}e^{-t}*\phi_{(\alpha,2)}(t)\\
\nonumber&=\int_{0}^{t}\sqrt{t-s}e^{-(t-s)}\phi_{(\alpha,2)}(s)ds\\
\nonumber&\leq {\cal{B}}(\frac{3}{2},\frac{1}{2})\int_{0}^{t}\sqrt{t-s}e^{-(t-s)}se^{-s}ds\\
&\leq {\cal{B}}(\frac{3}{2},\frac{1}{2})t\sqrt{t}e^{-t}.
\end{align}
For $\ell=3$, however,  we get the following estimate
\begin{align}\label{E4.45}
\nonumber\varpi_{(\alpha,3)}(t)&=\varpi_{(\alpha,2)}(t)*\phi_{(\alpha,3)}(t)\\
\nonumber&=\int_{0}^{t}\varpi_{(\alpha,2)}(s)\phi_{(\alpha,3)}(t-s)ds\\
\nonumber&\leq [{\cal{B}}(\frac{3}{2},\frac{1}{2})]^{2}\int_{0}^{t}s\sqrt{s}e^{-s}(t-s)e^{-2t+s}ds\\
&\leq [{\cal{B}}(\frac{3}{2},\frac{1}{2})]^{2}t^{2}\sqrt{t}e^{-t}.
\end{align}
Step by step, we obtain that
\begin{align}\label{E4.46}
\nonumber\varpi_{(\alpha,m)}(t)&=\varpi_{(\alpha,m-1)}(t)*\phi_{(\alpha,m)}(t)\\
\nonumber&=\int_{0}^{t}\varpi_{(\alpha,m-1)}(s)\phi_{(\alpha,m)}(t-s)ds\\
\nonumber&\leq [{\cal{B}}(\frac{3}{2},\frac{1}{2})]^{m-2}{\cal{B}}(\frac{3}{2},\frac{1}{2})
\int_{0}^{t}s^{m-1}\sqrt{s}e^{-s}(t-s)e^{(1-m)(t-s)}ds\\
\nonumber&\leq [{\cal{B}}(\frac{3}{2},\frac{1}{2})]^{m-1}\int_{0}^{t}s^{m-1}\sqrt{s}(t-s)e^{(1-m)t+(m-2)s}ds\\
&\leq [{\cal{B}}(\frac{3}{2},\frac{1}{2})]^{m-1}t^{m-1}\sqrt{t}e^{-t}.
\end{align}
By the induction principle, we conclude that
\begin{align}\label{E4.47}
\varpi_{(\alpha,m)}(t)\leq [{\cal{B}}(\frac{3}{2},\frac{1}{2})]^{m-1}t^{m-1}\sqrt{t}e^{-t},\hspace{0.2cm}m\geq2.
\end{align}
The above inequality proves that the Laplace transform of $\varpi_{(\alpha,m)}(t)$ is well defined. By using the formula \cite[p.92]{Sch}
\begin{align}\label{E4.48}
\mathscr{L}(f*g)=\mathscr{L}(f)\mathscr{L}(g),
\end{align}
combined  with the formula \cite[p.28]{Pru}
\begin{align}\label{E4.49}
\mathscr{L}(t^{a}e^{-bt})(j)=\frac{\Gamma(a+1)}{(j+b)^{a+1}},\hspace{0.2cm} a>-1,\hspace{0.2cm} b>0,
\end{align}
we get, by a direct computation, the following required equality
\begin{align}\label{E4.50}
\nonumber \mathscr{L}(\varpi_{(\alpha,m)}(t))(j)&=\mathscr{L}(\sqrt{t}e^{-t}*\sqrt{t}e^{-2t}*...*\sqrt{t}e^{-mt})(j)
\mathscr{L}(\frac{e^{-(\alpha+2)t}}{\sqrt{t}}*...*\frac{e^{-(\alpha+m)t}}{\sqrt{t}})(j)\\
&=\frac{(\Gamma(\frac{3}{2}))^{m}}{[(j+1)(j+2)...(j+m)]\frac{3}{2}}
\frac{(\Gamma(\frac{1}{2}))^{m-1}}{[(j+\alpha+2)(j+\alpha+3)...(j+\alpha+m)]\frac{1}{2}}.
\end{align}
\end{proof}
Now, we give the proof of the proposition $(\ref{P4.2}).$
\begin{proof}
\begin{description}
  \item[(of proposition (\ref{P4.2}))]
\end{description}
In order to construct the isometry given in proposition $(\ref{P4.2})$, we consider\\
$M_{1}=\mathbb{R}_{+},$ $d\mu_{1}(x)=x^{\alpha}e^{-x}dx$ and $\mathcal{A}=\mathcal{D}_{m}^{\alpha}$ where  $<,>_{\mathcal{A}}$ is the scalar product defined in $(\ref{E4.2})$.\\
As in  $(\ref{E2.47})$, the functions defined by
\begin{align}\label{E4.51}
\varphi_{j}(x)=\sqrt{\frac{j!}{\Gamma(\alpha+j+1)}}L_{j}^{(\alpha)}(x),\hspace{0.25cm}j\in\mathbb{Z}_{+},
\end{align}
where $L_{j}^{(\alpha)}$ is the Laguerre polynomials, defined in $(\ref{E2.46})$. This constitutes an orthonormal basis of the Hilbert space $L^{2}(\mathbb{R}_{+},x^{\alpha}e^{-x}dx).$\\
An orthonormal basis of the generalized Bergman-Dirichlet space $\mathcal{D}_{m}^{\alpha}$ can be given by the following family \cite{Elh}
 \begin{align}\label{E4.52}
   \psi_{j}(z)=\left\{
                 \begin{array}{ll}
                   \sqrt{\frac{\Gamma(j+1+\alpha)}{\pi j!\Gamma(\alpha+1)}}z^{j};&j<m, \hbox{} \\\\
                   \sqrt{\frac{(j-m)!\Gamma(j-m+2+\alpha)}{\pi(j!)^{2}\Gamma(\alpha+1)}}z^{j};&j\geq m. \hbox{}
                 \end{array}
               \right.
  \end{align}
By $(ii)$ of theorem $(\ref{T2.2})$, the following series
\begin{align}\label{E4.53}
   K(z,x)&=\displaystyle{\sum_{j\in\mathbb{Z}^{+}}}\overline{\varphi_{j}(x)}\psi_{j}(z)
  \end{align}
converges $\texttt{a.e}-d\mu_{1}(x)$ for each $z\in \mathbb{C}.$\\
By replacing the functions $\varphi_{j}(x)$ and $\psi_{j}(z)$ by its expressions given in $(\ref{E4.52})$ and $(\ref{E4.53})$, the kernel $K(z,x)$ can be rewritten as
\begin{align}\label{E4.54}
  \nonumber K(z,x)&=\frac{1}{\sqrt{\pi \Gamma(1+\alpha)}}\displaystyle{\sum_{0\leq j<m}}z^{j}L_{j}^{(\alpha)}(x)\\
   \nonumber &+\frac{1}{\sqrt{\pi \Gamma(1+\alpha)}}\displaystyle{\sum_{j=m}^{+\infty}}\sqrt{\frac{(j-m)!}{j!}}
   \sqrt{\frac{\Gamma(j-m+\alpha+2)}{\Gamma(j+\alpha+1)}}z^{j}L_{j}^{(\alpha)}(x)\\
   \nonumber&=\frac{1}{\sqrt{\pi \Gamma(1+\alpha)}}\displaystyle{\sum_{0\leq j<m}}z^{j}L_{j}^{(\alpha)}(x)\\
   &+\frac{1}{\sqrt{\pi \Gamma(1+\alpha)}}\displaystyle{\sum_{k=0}^{+\infty}}\sqrt{\frac{(k)!}{(k+m)!}}
   \sqrt{\frac{\Gamma(k+\alpha+2)}{\Gamma(k+m+\alpha+1)}}z^{k+m}L_{k+m}^{(\alpha)}(x).
  \end{align}
Applying the functional equation \cite[p.2]{Mag}
  \begin{align}\label{E4.55}
  \Gamma(z+\ell)=(z)_{\ell}\Gamma(z),\hspace{0.25cm}Re(z)>0\hspace{0.25cm}and\hspace{0.25cm}\ell\in\mathbb{Z}_{+}, \end{align}
   where $(z)_{\ell}=z(z+1)...(z+\ell-1),$ for $z=k+\alpha+2$ and $\ell=m-1,$ the equation $(\ref{E4.54})$ becomes
\begin{align}\label{E4.56}
  \nonumber  K(z,x)&=\frac{1}{\sqrt{\pi \Gamma(1+\alpha)}}\displaystyle{\sum_{0\leq j<m}}z^{j}L_{j}^{(\alpha)}(x)\\
   \nonumber&+\frac{z^{m}}{\sqrt{\pi \Gamma(1+\alpha)}}\displaystyle{\sum_{k=0}^{+\infty}}\frac{1}{\sqrt{(k+1)...(k+m)}}
   \frac{1}{\sqrt{(k+\alpha+2)...(k+\alpha+m)}}z^{k}L_{k+m}^{(\alpha)}(x)\\
   &=\frac{1}{\sqrt{\pi \Gamma(1+\alpha)}}\displaystyle{\sum_{0\leq j<m}}z^{j}L_{j}^{(\alpha)}(x)+\frac{z^{m}}{\sqrt{\pi \Gamma(1+\alpha)}}S(z,x),
   \end{align}
   where $S(z,x)$ is given by
   \begin{align}\label{E4.57}
 S(z,x)&=\displaystyle{\sum_{k=0}^{+\infty}}\frac{1}{\sqrt{(k+1)...(k+m)}}
   \frac{1}{\sqrt{(k+\alpha+2)...(k+\alpha+m)}}z^{k}L_{k+m}^{(\alpha)}(x).
   \end{align}
Note that (up to our knowledge) the above series does not appear in the literature as a standard closed generating formula. To avoid this problem, we first rewrite the series $S(z,x)$ as follows
\begin{align}\label{E4.58}
 S(z,x)&=m!\displaystyle{\sum_{k=0}^{+\infty}}\frac{1}{(k+1)^{\frac{3}{2}}...(k+m)^{\frac{3}{2}}}
   \frac{1}{(k+\alpha+2)^{\frac{1}{2}}...(k+\alpha+m)^{\frac{1}{2}}}\times \frac{(k+m)!}{m!k!}z^{k}L_{k+m}^{(\alpha)}(x).
   \end{align}
Secondly, from the point $(2)$ of lemma $(\ref{L4.1})$ we have
\begin{align}\label{E4.59}
\frac{1}{(k+1)^{\frac{3}{2}}
...(k+m)^{\frac{3}{2}}(k+\alpha+2)^{\frac{1}{2}}...(k+\alpha+m)^{\frac{1}{2}}}=
(\Gamma(\frac{3}{2}))^{-m}(\Gamma(\frac{1}{2}))^{-m+1}\mathscr{L}(\omega_{(\alpha,m)}(t))(k),
\end{align}
where  $\mathscr{L}$ is the classical Laplace transform defined in $(\ref{E4.40})$.\\
Thus, the equation $(\ref{E4.58})$ becomes
\begin{align}\label{E4.60}
 S(z,x)&=m!\displaystyle{\sum_{k=0}^{+\infty}}(\Gamma(\frac{3}{2}))^{-m}(\Gamma(\frac{1}{2}))^{1-m}
 \int_{0}^{+\infty}e^{-kt}\omega_{(\alpha,m)}(t)dt\frac{(k+m)!}{m!k!}z^{k}L_{k+m}^{(\alpha)}(x).
   \end{align}
To write the above sum in a closed formula, we consider the following series
\begin{align}\label{E4.61}
 S_{a}(z,x)&=\displaystyle{\sum_{k=0}^{+\infty}}
 \int_{a}^{+\infty}e^{-kt}\omega_{(\alpha,m)}(t)dt\frac{(k+m)!}{m!k!}z^{k}L_{k+m}^{(\alpha)}(x), \hspace{0.25cm}a\geq0.
   \end{align}
 Note that
 \begin{align}\label{E4.62}
 S(z,x)=m!(\Gamma(\frac{3}{2}))^{-m}(\Gamma(\frac{1}{2}))^{1-m}S_{0}(z,x).
 \end{align}
 Now, we will prove that the series $S_{a}(z,x)$ is convergent for all $a>0$. To do so, we use the following asymptotic formula given in $(\ref{E3.43})$
 \begin{align}\label{E4.63}
 L_{j}^{(\beta)}(x)=x^{-\frac{\beta}{2}-\frac{1}{4}}O(j^{\frac{\beta}{2}-\frac{1}{4}}),
 \end{align}
  as $j\longrightarrow+\infty$, for $\frac{c}{j}\leq x\leq w$ where $c$ and $w$ are fixed positive constants. Then, for $\beta=\alpha$ and $j=k+m$ enough large $k\geq k_{0}$ and by using $(\ref{E4.63})$, we obtain the following estimate
  \begin{align}\label{E4.64}
 \mid L_{k+m}^{(\alpha)}(x)\mid\leq M_{m}\mid x\mid^{-\frac{\alpha}{2}-\frac{1}{4}} (k+m)^{\frac{\alpha}{2}-\frac{1}{4}},\hspace{0.25cm}k\geq k_{0},
 \end{align}
 where $M_{m}$ is a positive constant.\\
 Next, taking $p\in\mathbb{Z}_{+}$ enough large and by using $(\ref{E4.64})$, we obtain the following inequality
\begin{align}\label{E4.65}
   \nonumber \mid\sum_{k=0}^{p}(ze^{-t})^{k}\frac{(k+m)!}{k!}L_{k+m}^{(\alpha)}(x)\mid&\leq \sum_{k=0}^{k_{0}-1}\frac{(k+m)!}{k!}\mid L_{k+m}^{(\alpha)}(x)\mid\\
   &+M_{m}\mid x\mid^{-\frac{\alpha}{2}-\frac{1}{4}}\sum_{k=k_{0}}^{+\infty}e^{-ak}\frac{(k+m)!}{k!}
   (k+m)^{\frac{\alpha}{2}-\frac{1}{4}}.
   \end{align}
Using the asymptotic formula
\begin{align}\label{E4.66}
\frac{\Gamma(z+a)}{\Gamma(z)}\sim z^{a}\hspace{0.2cm} when \hspace{0.2cm}\mid z\mid\longrightarrow+\infty,\hspace{0.2cm}a\in\mathbb{C},
\end{align}
for $z=k+1$ and $a=m$, we get the following behavior
\begin{align}\label{E4.67}
\frac{(k+m)!}{k!}=\frac{\Gamma(k+1+m)}{\Gamma(k+1)}\sim (1+k)^{m},\hspace{0.2cm}when\hspace{0.2cm}k\longrightarrow+\infty.
\end{align}
Returning back to $(\ref{E4.65})$ and using the above behavior, we get the following estimate
\begin{align}\label{E4.68}
   \nonumber \mid\sum_{k=0}^{p}(ze^{-t})^{k}\frac{(k+m)!}{k!}L_{k+m}^{(\alpha)}(x)\mid&\leq \sum_{k=0}^{k_{0}-1}\frac{(k+m)!}{k!}\mid L_{k+m}^{(\alpha)}(x)\mid\\
   \nonumber&+\tilde{M}_{m}\mid x\mid^{\frac{\alpha}{2}-\frac{1}{4}}\sum_{k=k_{0}}^{+\infty}e^{-ak}
   (k+1)^{m}(k+m)^{\frac{\alpha}{2}-\frac{1}{4}}\\
   &=C_{m,x}<+\infty,
   \end{align}
where $\tilde{M}_{m}$ and $C_{m,x}$ are positive constants.\\
Then, by using the Lebesgue dominated convergence theorem, we can interchange the sum and the integral given by $(\ref{E4.61})$
 in which we consider $a>0$. Then, we obtain the following equality
\begin{align}\label{E4.69}
S_{a}(z,x)=\int_{a}^{+\infty}\sum_{k=0}^{+\infty}(ze^{-t})^{k}
 \frac{(k+m)!}{k!m!}L_{k+m}^{(\alpha)}(x)\omega_{(\alpha,m)}(t)dt.
  \end{align}
 Next, applying the generating function \cite[p.104]{Mou}
\begin{align}\label{E4.70}
 \sum_{k=0}^{+\infty}\frac{(k+m)!}{m!k!}L^{(\beta)}_{(k+m)}(y)s^{k}=
(1-s)^{-\beta-m-1}\exp(\frac{-ys}{1-s})L^{(\beta)}_{(m)}(\frac{y}{1-s}),
\end{align}
for $\beta=\alpha$, $y=x$ and $s=ze^{-t}$, we find
 \begin{align}\label{E4.71}
 S_{a}(z,x)=\int_{a}^{+\infty}\omega_{(\alpha,m)}(t)(1-ze^{-t})^{-\alpha-m-1}
 \exp(\frac{-xze^{-t}}{1-ze^{-t}})L_{m}^{(\alpha)}(\frac{x}{1-ze^{-t}})dt.
 \end{align}
 We rewrite the above function $S_{a}(z,x)$ as
 \begin{align}\label{E4.72}
  S_{a}(z,x)=\int_{0}^{+\infty}1_{[a,+\infty[}\omega_{(\alpha,m)}(t)(1-ze^{-t})^{-\alpha-m-1}
 \exp(\frac{-xze^{-t}}{1-ze^{-t}})L_{m}^{(\alpha)}(\frac{x}{1-ze^{-t}})dt,
  \end{align}
 where $1_{[a,+\infty[}$ is the characteristic function associated with the set $[a,+\infty[\subset\mathbb{R}_{+}$.\\
 By the continuity of the functions $s\longmapsto e^{s}$ and $s\longmapsto L_{m}^{(\alpha)}(s)$, we can prove that the involved function in the right hand side of $(\ref{E4.72})$ satisfies the following estimated form
 \begin{align}\label{E4.73}
  \mid 1_{[a,+\infty[}\omega_{(\alpha,m)}(t)(1-ze^{-t})^{-\alpha-m-1}
 \exp(\frac{-xze^{-t}}{1-ze^{-t}})L_{m}^{(\alpha)}(\frac{x}{1-ze^{-t}})\mid\leq C_{z,x}^{m}t^{m}\sqrt{t}e^{-t},\hspace{0.25cm}C_{z,x}^{m}>0,
 \end{align}
  where we have used the following inequality
  \begin{align}\label{E4.74}
  \omega_{(\alpha,m)}(t)\leq ({\cal{B}}(\frac{3}{2},\frac{1}{2}))^{m-1}t^{m}\sqrt{t}e^{-t},
  \end{align}
  given in the first point of the lemma $(\ref{L4.1})$.\\
  Finally, applying the Lebesgue dominated convergence theorem, we get
\begin{align}\label{E4.75}
  \lim_{a\longrightarrow 0}S_{a}(z,x)=\int_{0}^{+\infty}\omega_{(\alpha,m)}(t)(1-ze^{-t})^{-\alpha-m-1}
 \exp(\frac{-xze^{-t}}{1-ze^{-t}})L_{m}^{(\alpha)}(\frac{x}{1-ze^{-t}})dt.
  \end{align}
Hence, we have
 \begin{align}\label{E4.76}
 S(z,x):=S_{0}(z,x)=\int_{0}^{+\infty}\omega_{(\alpha,m)}(t)(1-ze^{-t})^{-\alpha-m-1}
 \exp(\frac{-xze^{-t}}{1-ze^{-t}})L_{m}^{(\alpha)}(\frac{x}{1-ze^{-t}})dt.
  \end{align}
Returning back to the equation $(\ref{E4.56})$, we get the desired kernel
\begin{align}\label{E4.77}
  \nonumber K(z,x)&=\frac{1}{\sqrt{\pi \Gamma(1+\alpha)}}\displaystyle{\sum_{0\leq j<m}}z^{j}L_{j}^{(\alpha)}(x)\\
   &\hspace{-1.5cm}+\frac{m!z^{m}(\Gamma(\frac{3}{2}))^{-m}\Gamma(\frac{1}{2}))^{1-m}}{\sqrt{\pi \Gamma(1+\alpha)}}
   \int_{0}^{+\infty}\omega_{(\alpha,m)}(t)(1-ze^{-t})^{-\alpha-m-1}
 \exp(\frac{-xze^{-t}}{1-ze^{-t}})L_{m}^{(\alpha)}(\frac{x}{1-ze^{-t}})dt.
   \end{align}
   This ends the proof.
\end{proof}
\begin{remark}
By considering the following natural isometry
\begin{align}\label{E4.78}
		\nonumber T_{\alpha}:\hspace{0.25cm}L^{2}(\mathbb{R}_{+},dx)&\longrightarrow L^{2}(\mathbb{R}_{+},x^{\alpha}e^{-x}dx)\\
		&f\longmapsto T_{\alpha}(f)(x)=x^{\frac{-\alpha}{2}}e^{\frac{x}{2}}f(x),
\end{align}
we obtain, in a canonical way from the last proposition, the following isometry
\begin{align}\label{E4.79}
 \nonumber L^{2}(\mathbb{R}_{+},& dx)\longrightarrow \mathcal{D}_{m}^{\alpha}\\
 &f\longmapsto B[f](z):=\int_{0}^{+\infty}K(z,x)f(x)dx,
 \end{align}
  where the kernel $K(z,x)$ is given by
  \begin{align}\label{E4.80}
  \nonumber K(z,x)&=\frac{x^{\frac{-\alpha}{2}}e^{\frac{x}{2}}}{\sqrt{\pi \Gamma(1+\alpha)}}\displaystyle{\sum_{0\leq j<m}}z^{j}L_{j}^{(\alpha)}(x)\\
   &\hspace{-1.8cm}+\frac{m!z^{m}(\Gamma(\frac{3}{2}))^{-m}\Gamma(\frac{1}{2}))^{1-m}x^{\frac{-\alpha}{2}}
   e^{\frac{x}{2}}}{\sqrt{\pi \Gamma(1+\alpha)}}
   \int_{0}^{+\infty}\omega_{(\alpha,m)}(t)(1-ze^{-t})^{-\alpha-m-1}
 \exp(\frac{-xze^{-t}}{1-ze^{-t}})L_{m}^{(\alpha)}(\frac{x}{1-ze^{-t}})dt.
   \end{align}
\end{remark}
\section*{Acknowledgements} The assistance of the members of the seminary "Spectral theory and quantum dynamics" is gratefully acknowledged, specially Adil Belhaj for his helpful discussion.


\begin{thebibliography}{10}

\bibitem{Ara2}
J.~Arazy, S.~D. Fisher, et~al.
\newblock The uniqueness of the dirichlet space among mobius-invariant hilbert
  spaces.
\newblock {\em Illinois Journal of Mathematics}, 29(3):449--462, 1985.

\bibitem{Ara1}
J.~Arazy, S.~Janson, and J.~Peetre.
\newblock An identity for reproducing kernels in a planar domain and the
  hubert-schmidt hankel operators.
\newblock {\em J. Reine Angew. Alatli}, 406:179--199, 1990.

\bibitem{Arc}
N.~Arcozzi, R.~Rochberg, E.~T. Sawyer, and B.~D. Wick.
\newblock The dirichlet space: a survey.
\newblock {\em New York J. Math. A}, 17:45--86, 2011.

\bibitem{Aro}
N.~Aronszajn.
\newblock Theory of reproducing kernels.
\newblock {\em Transactions of the American mathematical society},
  68(3):337--404, 1950.

\bibitem{Bar}
V.~Bargmann.
\newblock On a hilbert space of analytic functions and an associated integral
  transform part \textsc{I}.
\newblock {\em Communications on pure and applied mathematics}, 14(3):187--214,
  1961.

\bibitem{Ber}
H.~Bercovici, A.~Brown, and C.~Pearcy.
\newblock {\em Measure and Integration}.
\newblock Springer, 2016.

\bibitem{Elf}
O.~El-Fallah, K.~Kellay, J.~Mashreghi, and T.~Ransford.
\newblock {\em A primer on the Dirichlet space}, volume 203.
\newblock Cambridge University Press, 2014.

\bibitem{Elh}
A.~El~Hamyani, A.~Ghanmi, A.~Intissar, Z.~Mouhcine, and M.~S. El~A{\"\i}nin.
\newblock Generalized weighted bergman--dirichlet and bargmann--dirichlet
  spaces: explicit formulae for reproducing kernels and asymptotics.
\newblock {\em Annals of Global Analysis and Geometry}, 49(1):59--72, 2016.

\bibitem{Elk}
A.~El~Kachkouri and A.~Ghanmi.
\newblock The slice hyperholomorphic bergman space on $\mathbb{B}_{R}$:
  Integral representation and asymptotic behavior.
\newblock {\em Complex Analysis and Operator Theory}, 12(5):1351--1367, 2018.

\bibitem{Elw}
F.~El~Wassouli, A.~Ghanmi, A.~Intissar, and Z.~Mouayn.
\newblock Generalized second {B}argmann transforms associated with the
  hyperbolic {L}andau levels in the {P}oincar\'{e} disk.
\newblock {\em Ann. Henri Poincar\'{e}}, 13(3):513--524, 2012.

\bibitem{Far}
J.~Faraut.
\newblock {\em Analysis on Lie groups: an introduction}, volume 110.
\newblock Cambridge University Press, 2008.

\bibitem{Fay}
J.~D. Fay.
\newblock Fourier coefficients of the resolvent for a fuchsian group.
\newblock {\em J. reine angew. Math}, 293(294):143--203, 1977.

\bibitem{Hal4}
B.~C. Hall.
\newblock The segal-bargmann" coherent state" transform for compact lie groups.
\newblock {\em Journal of functional analysis}, 122(1):103--151, 1994.

\bibitem{Hal5}
B.~C. Hall.
\newblock The inverse segal--bargmann transform for compact lie groups.
\newblock {\em journal of functional analysis}, 143(1):98--116, 1997.

\bibitem{Hal1}
B.~C. Hall and J.~J. Mitchell.
\newblock The segal--bargmann transform for noncompact symmetric spaces of the
  complex type.
\newblock {\em Journal of Functional Analysis}, 227(2):338--371, 2005.

\bibitem{Hal2}
B.~C. Hall, J.~J. Mitchell, et~al.
\newblock The segal-bargmann transform for compact quotients of symmetric
  spaces of the complex type.
\newblock {\em Taiwanese Journal of Mathematics}, 16(1):13--45, 2012.

\bibitem{Hed}
H.~Hedenmalm, B.~Korenblum, and K.~Zhu.
\newblock Theory of bergman spaces, vol. 199 of graduate texts in mathematics,
  2000.

\bibitem{Hep}
H.~Liu and L.~Peng.
\newblock Weighted plancherel formula. irreducible unitary representations and
  eigenspace representations.
\newblock {\em Mathematica Scandinavica}, pages 99--119, 1993.

\bibitem{Mag}
W.~Magnus, F.~Oberhettinger, and R.~P. Soni.
\newblock {\em Formulas and theorems for the special functions of mathematical
  physics}, volume~52.
\newblock Springer Science \& Business Media, 2013.

\bibitem{Val}
V.~Moretti.
\newblock {\em Spectral Theory and Quantum Mechanics: With an Introduction to
  the Algebraic Formulation}.
\newblock Unitext / La Matematica per il 3+2. Springer, 2013 edition, 2013.

\bibitem{Zou3}
Z.~Mouayn.
\newblock Coherent state transforms attached to generalized bargmann spaces on
  the complex plane.
\newblock {\em Mathematische Nachrichten}, 284(14-15):1948--1954, 2011.

\bibitem{Zou1}
Z.~Mouayn.
\newblock Une famille de transformations de bargmann circulaires.
\newblock {\em Comptes Rendus Mathematique}, 350(23-24):1017--1022, 2012.

\bibitem{Zou4}
Z.~Mouayn.
\newblock Polyanalytic relativistic second bargmann transforms.
\newblock {\em Journal of Mathematical Physics}, 56(5):053501, 2015.

\bibitem{Zou2}
Z.~Mouayn.
\newblock Transformations de bargmann discr{\`e}tes attach{\'e}es aux niveaux
  de landau sur la sph{\`e}re de riemann.
\newblock In {\em Annales Henri Poincar{\'e}}, volume~16, pages 641--650.
  Springer, 2015.

\bibitem{Mou}
I.~Mourad.
\newblock {\em Classical and Quantum Orthogonal Polynomials in One Variable}.
\newblock Cambridge University Press, 2005.

\bibitem{Ola}
G.~\'{O}lafsson and B.~{\O}rsted.
\newblock {\em Generalizations of the Bargmann transform}.
\newblock Citeseer, 1996.

\bibitem{Pau}
V.~I. Paulsen and M.~Raghupathi.
\newblock {\em An introduction to the theory of reproducing kernel Hilbert
  spaces}, volume 152.
\newblock Cambridge University Press, 2016.

\bibitem{Pru}
A.~Prudnikov, Y.~Brychkov, and O.~Marichev.
\newblock {\em Integrals and Series. Vol. 4: Direct Laplace Transforms. Vol. 5:
  Inverse Laplace Transforms}.
\newblock CRC Pub., Apr, 1992.

\bibitem{Ran}
E.~Rainville.
\newblock Special functions macmillan company, new york, 1960.

\bibitem{Bea}
R.~W. Richard~Beals.
\newblock {\em Special Functions: A Graduate Text}.
\newblock Cambridge Studies in Advanced Mathematics. Cambridge University
  Press, 2010.

\bibitem{Sch}
J.~L. Schiff.
\newblock {\em The Laplace transform: theory and applications}.
\newblock Springer Science \& Business Media, 2013.

\bibitem{Kon}
K.~Schm{\"u}dgen.
\newblock {\em Unbounded self-adjoint operators on Hilbert space}, volume 265.
\newblock Springer Science \& Business Media, 2012.

\bibitem{Ste}
M.~B. Stenzel.
\newblock The segal--bargmann transform on a symmetric space of compact type.
\newblock {\em Journal of Functional Analysis}, 165(1):44--58, 1999.

\bibitem{Wu}
Z.~Wu.
\newblock Function theory and operator theory on the dirichlet space.
\newblock {\em Holomorphic Spaces, Math. Sci. Res. Inst.(Berkeley), Publ},
  33:179--199, 1998.

\bibitem{Zha}
G.~Zhang.
\newblock {\em A weighted Plancherel formula II: The case of the ball}.
\newblock Institut Mittag-Leffler, 1991.

\bibitem{Zhu2}
K.~Zhu.
\newblock {\em Spaces of holomorphic functions in the unit ball}, volume 226.
\newblock Springer Science \& Business Media, 2005.

\bibitem{Zhu1}
K.~Zhu.
\newblock {\em Analysis on Fock spaces}, volume 263.
\newblock Springer Science \& Business Media, 2012.

\end{thebibliography}
\end{document}